\documentclass[a4paper,12pt]{amsart}
\usepackage{amsmath,amsfonts,amssymb,latexsym}
\usepackage{hyperref}
\newcommand{\RR}{\mathbb{R}}
\newcommand{\CC}{\mathbb{C}}
\newcommand{\QQ}{\mathbb{Q}}
\newcommand{\NN}{\mathbb{N}}
\newcommand{\ZZ}{\mathbb{Z}}
\newcommand{\EE}{\mathbb{E}}
\newcommand{\Bo}{\mathcal{B}}
\newcommand{\Oo}{\mathcal{O}}

\newcommand{\RE}{ {\rm Re \,} }

\newtheorem{Th}{Theorem}
\newtheorem{Lem}{Lemma}

\newtheorem{Prop}{Proposition}

\theoremstyle{definition}
\newtheorem{Df}{Definition}

\theoremstyle{remark}
\newtheorem{Rem}{Remark}
\newtheorem{Ex}{Example}

\begin{document}

\keywords{linear PDEs with constant coefficients,
formal power series, moment functions, moment-PDEs, Gevrey order, Borel summability, multisummability.}

\subjclass[2010]{35C10, 35C15, 35E15, 40G10.}
\title[Inhomogeneous moment PDEs]{Analytic and summable solutions of inhomogeneous moment partial
differential equations}

\author{S{\l}awomir Michalik}

\address{Faculty of Mathematics and Natural Sciences,
College of Science\\
Cardinal Stefan Wyszy\'nski University\\
W\'oycickiego 1/3,
01-938 Warszawa, Poland}
\email{s.michalik@uksw.edu.pl}
\urladdr{\url{http://www.impan.pl/~slawek}}

\begin{abstract}
We study the Cauchy problem for a general inhomogeneous linear
moment partial differential equation of two complex
variables with constant coefficients, where the inhomogeneity is given by the formal power series.
We state sufficient conditions for the convergence, analytic continuation and summability
of formal power series solutions in terms of properties of the inhomogeneity.
We consider both the summability in one variable $t$ (with coefficients belonging
to some Banach space of Gevrey series with respect to the second variable $z$)
and the summability in two variables $(t,z)$.
\end{abstract}

\maketitle

\section{Introduction}
The formal $m$-moment differentiation $\partial_{m,z}$ is a linear operator on the space of formal power series defined by
$$
\partial_{m,z}\Big(\sum_{j=0}^{\infty}\frac{u_j z^{j}}{m(j)}\Big):=
\sum_{j=0}^{\infty}\frac{u_{j+1}z^{j}}{m(j)},
$$
where $m(u)$ is a moment function constructed in the Balser theory of moment summability \cite[Section 5.5]{B2}. 
In the special case $m(u)=\Gamma(1+u)$, the operator $\partial_{m,z}$ coincides with the usual differentiation $\partial_z$.
More generally, for $p\in\NN$ and $m(u)=\Gamma(1+u/p)$ the operator $\partial_{m,z}$ is closely related to the $1/p$-fractional
differentiation $\partial^{1/p}_z$.

The concept of moment differentiation was introduced by Balser and Yoshino \cite{B-Y}, who studied the Gevrey order
of formal solutions of inhomogeneous moment partial differential equations with constant coefficients. 

The theory of moment partial differential
equations was developed by the author \cite{Mic7,Mic8}. In \cite{Mic7} formal, analytic and summable solutions of homogeneous
moment partial differential equations were studied.
In the same paper the author also calculated normalised formal solutions in the inhomogeneous case and he determined their Gevrey orders.
Next, in \cite{Mic8} the author extended the results of \cite{Mic7} to homogeneous equations with divergent Cauchy data.
Moreover, Lastra, Malek and Sanz \cite{La-Ma-Sa} generalised the results of \cite{Mic7} about homogeneous moment equations to the case
when moment functions are determined by 
given strongly regular moment sequences.

The present paper is a natural generalisation of \cite{Mic7,Mic8} to the inhomogeneous case with the inhomogeneity given by the formal power series.
More precisely, we consider the initial value problem for a general inhomogeneous linear moment partial differential equation
of two complex variables $(t,z)$ with constant coefficients
\begin{gather}
\label{eq:general_p}
P(\partial_{m_1,t},\partial_{m_2,z})\widehat{u}=\widehat{f}(t,z),\quad \partial^j_{m_1,t} \widehat{u}(0,z)=0\ \textrm{for}\ j=0,\dots,n-1,
\end{gather}
where $P(\lambda,\zeta)$ is a polynomial of two variables of degree $n$ with respect to $\lambda$ and
the inhomogeneity $\widehat{f}(t,z)$ is a formal power series in both variables of Gevrey order $(\tilde{s}_1,\tilde{s}_2)\in\RR^2$,
and $m_1$, $m_2$ are moment functions of orders $s_1$, $s_2$ respectively.
\par
Similarly to \cite{Mic8}, we consider the wider class of  moment functions than in \cite{B-Y} or \cite{Mic7}, which creates the group with respect to multiplication.
This extension allows us to study some integro-differential operators as moment differential operators.

As in the previous papers \cite{Mic7,Mic8} we use the moment pseudodifferential operators $\lambda_i(\partial_{m_2,z})$ to factorise
the operator $P$ as
$$
P(\partial_{m_1,t},\partial_{m_2,z})=P_0(\partial_{m_2,z})(\partial_{m_1,t}-\lambda_1(\partial_{m_2,z}))^{n_1}\dots(\partial_{m_1,t}-\lambda_l(\partial_{m_2,z}))^{n_l},
$$
where $P_0(\zeta)$ is a polynomial and $\lambda_1(\zeta),\dots,\lambda_l(\zeta)$ are
the roots of the characteristic equation $P(\lambda,\zeta)=0$ of (\ref{eq:general_p}) with multiplicities $n_1,\dots,n_l$ ($n_1+\cdots+n_l=n$) respectively.

In general, a formal solution of (\ref{eq:general_p}) is not unique,
but it is uniquely determined by every formal power series $\widehat{g}$ satisfying $P_0(\partial_{m_2,z})\widehat{g}=\widehat{f}$,
since there is exactly one formal solution of (\ref{eq:general_p}) satisfying also
the moment pseudodifferential equation
$$
(\partial_{m_1,t}-\lambda_1(\partial_{m_2,z}))^{n_1}\cdots(\partial_{m_1,t}-\lambda_l(\partial_{m_2,z}))^{n_l}\widehat{u}
=\widehat{g}.
$$
If $\widehat{u}$ is a formal solution of (\ref{eq:general_p}) determined by $\widehat{g}$ then
$\widehat{u}=\sum_{\alpha=1}^l\sum_{\beta=1}^{n_{\alpha}}\widehat{u}_{\alpha\beta}$, where $\widehat{u}_{\alpha\beta}$ satisfies the simple
moment pseudodifferential equation
\begin{gather}
\label{eq:simply}
(\partial_{m_1,t}-\lambda_\alpha(\partial_{m_2,z}))^{\beta} \widehat{u}_{\alpha\beta}=
\widehat{g}_{\alpha\beta},\quad \partial^j_{m_1,t} \widehat{u}_{\alpha\beta}(0,z)=0\
\textrm{for}\ j=0,\dots,\beta-1,
\end{gather}
for some formal series $\widehat{g}_{\alpha\beta}$ constructed from $\widehat{g}$.

In this way we reduce the problem (\ref{eq:general_p}) to (\ref{eq:simply}). 
We immediately calculate the Gevrey order of $\widehat{u}_{\alpha\beta}$, which depends on a pole order of $\lambda_\alpha(\zeta)$,
the Gevrey order of $\widehat{g}$ and
orders of moment functions $m_1$ and $m_2$ (Theorem \ref{th:gevrey}).   

If the Gevrey order of $\widehat{u}_{\alpha\beta}$ is equal to zero, we study the analytic continuation properties of $u_{\alpha\beta}$.
To this end we observe that by the general
theory of moment summability, without loss of generality
we may assume that $m_1(u)=\Gamma(1+s_1u)$ and $m_2(u)=\Gamma(1+s_2u)$. In this case we construct the integral representation of $u_{\alpha\beta}$,
which allows
us to find the connection between analytic continuation properties of solution $u_{\alpha\beta}$ and inhomogeneity $g_{\alpha\beta}$
(Theorem \ref{th:2}).

If the Gevrey order of $\widehat{u}_{\alpha\beta}$ is greater than zero, then applying appropriate moment Borel transforms $\Bo_{m_1',t}$ and
$\Bo_{m_2',z}$ to (\ref{eq:simply}) we transform the formal solution $\widehat{u}_{\alpha\beta}$ of (\ref{eq:simply}) with divergent inhomogeneity
$\widehat{g}_{\alpha\beta}$ into the analytic solution 
$v_{\alpha\beta}=\Bo_{m_1',t}\Bo_{m_2',z}\widehat{u}_{\alpha\beta}$ of the equation
$$
(\partial_{m_1m_1',t}-\lambda_\alpha(\partial_{m_2m_2',z}))^{\beta} v_{\alpha\beta}=\Bo_{m_1',t}\Bo_{m_2',z}\widehat{g}_{\alpha\beta}
$$
with the convergent inhomogeneity $\Bo_{m_1',t}\Bo_{m_2',z}\widehat{g}_{\alpha\beta}$.
In this way we obtain the sufficient conditions
for the summability of $\widehat{u}_{\alpha\beta}$ (both in one variable $t$ and in two variables $(t,z)$) in terms of the analytic continuation
properties of the inhomogeneity (Theorem \ref{th:simple_sum}).
\par
Finally, returning to the general equation (\ref{eq:general_p}), we get the sufficient condition for the multisummability of $\widehat{u}$
in terms of the inhomogeneity $\widehat{g}$
(Theorem \ref{th:multi1}).
\par
In the last section we define the Newton polygon for moment partial differential operators with constant coefficients and we use
it to describe the multisummable solutions of
(\ref{eq:general_p}) in the case when inhomogeneity $f$ is analytic in some complex neighbourhood of the origin.
\par
Since in the special case $m_1(u)=m_2(u)=\Gamma(1+u)$, (\ref{eq:general_p}) is the Cauchy problem for an inhomogeneous linear partial differential
equation with constant coefficients and with a divergent inhomogeneity, the paper provides new sufficient conditions for the summability
of formal solutions of such initial value problem. In this way the results of this paper extends our knowledge about summability of formal solutions
of inhomogeneous linear partial differential equations, which was studied earlier by such authors as
Balser \cite{B4}, Balser and Loday-Richaud \cite{B-L}, Balser, Duval and Malek \cite{B-D-M}, Michalik \cite{Mic2,Mic6} and 
Tahara and Yamazawa \cite{T-Y}.

\section{Notation}
Throughout this paper we use the following notation.
The complex disc in $\CC^n$ with centre at the origin
and radius $r>0$ is denoted by $D^n_r:=\{z\in\CC^n:\ |z|< r\}$.
To simplify notation, we write $D_r$ instead of $D^1_r$. If the radius $r$ is not essential,
then we denote it briefly by $D^n$ (resp. $D$).
\par
\emph{A sector in a direction $d\in\RR$ with an opening $\varepsilon>0$}
in the universal covering space $\widetilde{\CC\setminus\{0\}}$ of $\CC\setminus\{0\}$ is defined by
\[
S_d(\varepsilon):=\{z\in\widetilde{\CC\setminus\{0\}}:\ z=re^{i\theta},\
d-\varepsilon/2<\theta<d+\varepsilon/2,\ r>0\}.
\]
Moreover, if the value of opening angle $\varepsilon$ is not essential, then we denote it briefly by $S_d$.
\par
Analogously, by \emph{a disc-sector in a direction $d\in\RR$ with an opening $\varepsilon>0$ and radius $r>0$}
we mean a domain $\hat{S}_d(\varepsilon;r):=S_d(\varepsilon)\cup D_r$. If the values of $\varepsilon$ and $r$
are not essential, we write it $\hat{S}_d$ for brevity (i.e. $\hat{S}_d=S_d\cup D$).
\par
By $\mathcal{O}(G)$ we understand the space of holomorphic functions on a domain $G\subseteq\CC^n$.
Analogously,
the space of analytic functions of the variables $z_1^{1/\kappa_{1}},\dots,z_n^{1/\kappa_n}$
($(\kappa_1,\dots,\kappa_n)\in\NN^n$) on $G$ is denoted by $\mathcal{O}_{1/\kappa_{1},\dots,1/\kappa_n}(G)$.
More generally, if $\EE$ denotes a Banach space with a norm $\|\cdot\|_{\EE}$, then
by $\Oo(G,\EE)$ (resp. $\Oo_{1/\kappa_{1},\dots,1/\kappa_n}(G,\EE)$) we shall denote the set of all $\EE$-valued  
holomorphic functions (resp. holomorphic functions of the variables $z_1^{1/\kappa_{1}},\dots,z_n^{1/\kappa_n}$) 
on a domain $G\subseteq\CC^n$.
For more information about functions with values in Banach spaces we refer the reader to \cite[Appendix B]{B2}. 
In the paper, as a Banach space $\EE$ we will
take the space of complex numbers $\CC$ (we abbreviate $\Oo(G,\CC)$ to $\Oo(G)$ and
$\Oo_{1/\kappa_{1},\dots,1/\kappa_n}(G,\CC)$ to $\Oo_{1/\kappa_{1},\dots,1/\kappa_n}(G)$)
or the space of Gevrey series $G_{s,1/\kappa}(r)$ (see Definition \ref{df:G_s}).
\par
\begin{Df}
A function $u\in\Oo_{1/\kappa}(\hat{S}_d(\varepsilon;r),\EE)$
is of \emph{exponential growth of order at most $K\in\RR$ as $x\to\infty$
in $\hat{S}_d(\varepsilon;r)$} if for any
$\widetilde{\varepsilon}\in(0,\varepsilon)$ and $\widetilde{r}\in(0,r)$ there exist 
$A,B<\infty$ such that
\begin{gather*}
\|u(x)\|_{\EE}<Ae^{B|x|^K} \quad \textrm{for every} \quad x\in \hat{S}_d(\widetilde{\varepsilon};\widetilde{r}).
\end{gather*}
The space of such functions is denoted by $\Oo_{1/\kappa}^K(\hat{S}_d(\varepsilon;r),\EE)$.
\par
Analogously, a function $u\in\Oo_{1/\kappa_1,1/\kappa_2}(\hat{S}_{d_1}(\varepsilon_1;r_1)\times\hat{S}_{d_2}(\varepsilon_2;r_2))$
is of \emph{exponential growth of order at most $(K_1,K_2)\in\RR^2$ as $(t,z)\to\infty$
in $\hat{S}_{d_1}(\varepsilon_1;r_1)\times\hat{S}_{d_2}(\varepsilon_2;r_2)$} if
for any $\widetilde{\varepsilon}_i\in(0,\varepsilon_i)$ and any $\widetilde{r}_i\in(0,r_i)$ ($i=1,2$)
there exist $A,B_1,B_2<\infty$ such that
\begin{gather*}
|u(t,z)|<Ae^{B_1|t|^{K_1}}e^{B_2|z|^{K_2}} \quad \textrm{for every} \quad (t,z)\in 
\hat{S}_{d_1}(\widetilde{\varepsilon}_1;\widetilde{r}_1)\times
\hat{S}_{d_2}(\widetilde{\varepsilon}_2;\widetilde{r}_2).
\end{gather*}
The space of such functions is denoted by
$\Oo^{K_1,K_2}_{1/\kappa_1,1/\kappa_2}(\hat{S}_{d_1}(\varepsilon_1;r_1)\times\hat{S}_{d_2}(\varepsilon_2;r_2) )$.
\end{Df}
\par
The space of formal power series
$ \widehat{u}(x)=\sum_{j=0}^{\infty}u_j x^{j/\kappa}$ with $u_j\in\mathbb{E}$ is denoted by $\mathbb{E}[[x^{\frac{1}{\kappa}}]]$.
Analogously, the space of formal power series $\widehat{u}(t,z)=\sum_{j,n=0}^{\infty}u_{jn} t^{j/\kappa_1}
z^{n/\kappa_2}$ with
$u_{jn}\in\mathbb{E}$ is denoted by $\mathbb{E}[[t^{\frac{1}{\kappa_1}},z^{\frac{1}{\kappa_2}}]]$.

We use the ``hat'' notation ($\widehat{u}$, $\widehat{v}$, $\widehat{\varphi}$, $\widehat{f}$, $\widehat{g}$) to denote the formal power series.
If the formal power series $\widehat{u}$ (resp. $\widehat{v}$, $\widehat{\varphi}$, $\widehat{f}$, $\widehat{g}$) is convergent,
we denote its sum by $u$ (resp. $v$, $\varphi$, $f$, $g$).

\section{Moment functions}
   In this section we recall the notion of moment methods introduced by Balser \cite{B2}.
   
   \begin{Df}[see {\cite[Section 5.5]{B2}}]
    \label{df:moment}
    A pair of functions $e_m$ and $E_m$ is said to be \emph{kernel functions of order $k$} ($k>1/2$) if
    they have the following properties:
   \begin{enumerate}
    \item[1.] $e_m\in\Oo(S_0(\pi/k))$, $e_m(z)/z$ is integrable at the origin, $e_m(x)\in\RR_+$ for $x\in\RR_+$ and
     $e_m$ is exponentially flat of order $k$ as $z\to\infty$ in $S_0(\pi/k)$ (i.e. $\forall_{\varepsilon > 0} \exists_{A,B > 0}$
     such that $|e_m(z)|\leq A e^{-(|z|/B)^k}$ for $z\in S_0(\pi/k-\varepsilon)$).
    \item[2.] $E_m\in\Oo^{k}(\CC)$ and $E_m(1/z)/z$ is integrable at the origin in $S_{\pi}(2\pi-\pi/k)$.
    \item[3.] The connection between $e_m$ and $E_m$ is given by the \emph{corresponding moment function
    $m$ of order $1/k$} as follows.
     The function $m$ is defined by the Mellin transform of $e_m$
     \begin{gather}
      \label{eq:e_m}
      m(u):=\int_0^{\infty}x^{u-1} e_m(x)dx \quad \textrm{for} \quad \RE u \geq 0
     \end{gather}
     and the kernel function $E_m$ has the power series expansion
     \begin{gather}
      \label{eq:E_m}
      E_m(z)=\sum_{n=0}^{\infty}\frac{z^n}{m(n)} \quad  \textrm{for} \quad z\in\CC.
     \end{gather}
   \end{enumerate}
   \end{Df}
   
   \begin{Rem}
    Observe that by the inverse Mellin transform and by (\ref{eq:E_m}), the moment function $m$
    uniquely determines the kernel functions $e_m$ and $E_m$.
   \end{Rem}

    In case $k\leq 1/2$ the set $S_{\pi}(2\pi-\pi/k)$ is not defined,
    so the second property in Definition \ref{df:moment} can not be satisfied. It means that we
    must define the kernel functions of order $k\leq 1/2$ and the corresponding moment functions
    in another way.
    
    \begin{Df}[see {\cite[Section 5.6]{B2}}]
     \label{df:small}
     A function $e_m$ is called \emph{a kernel function of order $k>0$} if we
     can find a pair of kernel functions $e_{\widetilde{m}}$ and $E_{\widetilde{m}}$ of
     order $pk>1/2$ (for some $p\in\NN$) so that
     \begin{gather*}
      e_m(z)=e_{\widetilde{m}}(z^{1/p})/p \quad \textrm{for} \quad z\in S_0(\pi/k).
     \end{gather*}
     For a given kernel function $e_m$ of order $k>0$ we define the
     \emph{corresponding moment function $m$ of order $1/k>0$} by (\ref{eq:e_m}) and
     the \emph{kernel function $E_m$ of order $k>0$} by (\ref{eq:E_m}).
    \end{Df}
    
    \begin{Rem}
     Observe that by Definitions \ref{df:moment} and \ref{df:small} we have
     \begin{eqnarray*}
      m(u)=\widetilde{m}(pu) & \textrm{and} &
      E_m(z)=\sum_{j=0}^{\infty}\frac{z^j}{m(j)}=\sum_{j=0}^{\infty}\frac{z^j}{\widetilde{m}(jp)}.
     \end{eqnarray*}
    \end{Rem}

As in \cite{Mic8}, we extend the notion of moment functions to real orders.
\begin{Df}
 \label{df:moment_general}
     We say that $m$ is a \emph{moment function of order $1/k<0$} if $1/m$ is a moment function of order $-1/k>0$.
     \par
     We say that $m$ is a \emph{moment function of order $0$} if there exist moment functions $m_1$ and $m_2$ of the same order $1/k>0$ such that $m=m_1/m_2$.
\end{Df}

By Definition \ref{df:moment_general} and by \cite[Theorems 31 and 32]{B2} we have
\begin{Prop}
 Let $m_1$, $m_2$ be moment functions of orders $s_1,s_2\in\RR$ respectively. Then
 \begin{itemize}
  \item $m_1m_2$ is a moment function of order $s_1+s_2$,
  \item $m_1/m_2$ is a moment function of order $s_1-s_2$.
 \end{itemize}
\end{Prop}

\begin{Ex}
\label{ex:functions}
 For any $a\geq 0$, $b\geq 1$ and $k>0$ we can construct the following examples of kernel functions $e_m$ and 
 $E_m$ of 
orders $k>0$ with the corresponding moment function $m$ of order $1/k$ satisfying Definition \ref{df:moment}
or \ref{df:small}:
 \begin{itemize}
  \item $e_m(z)=akz^{bk}e^{-z^k}$,
  \item $m(u)=a\Gamma(b+u/k)$,
  \item $E_m(z)=\frac{1}{a}\sum_{j=0}^{\infty}\frac{z^j}{\Gamma(b+j/k)}$.
\end{itemize}
In particular for $a=b=1$ we get the kernel functions and the corresponding moment function, which are used  
in the classical theory of $k$-summability.
\begin{itemize}
    \item $e_m(z)=kz^ke^{-z^k}$,
    \item $m(u)=\Gamma(1+u/k)$,
    \item $E_m(z)=\sum_{j=0}^{\infty}\frac{z^j}{\Gamma(1+j/k)}=:\mathbf{E}_{1/k}(z)$, where $\mathbf{E}_{1/k}$ is the
    Mittag-Leffler function of index $1/k$.
\end{itemize}
\end{Ex}

\begin{Ex}
For any $s\in\RR$ we will denote by $\Gamma_s$ the function
\[
  \Gamma_s(u):=\left\{
  \begin{array}{lll}
    \Gamma(1+su) & \textrm{for} & s \geq 0\\
    1/\Gamma(1-su) & \textrm{for} & s < 0.
  \end{array}
  \right.
\]
Observe that by Example \ref{ex:functions} and Definition \ref{df:moment_general}, $\Gamma_s$
is an example of a moment function of order $s\in\RR$.
\end{Ex}

The moment functions $\Gamma_s$ will be extensively used in the paper,
since every moment function $m$ of order $s$ has the same growth as $\Gamma_s$. Precisely speaking,
we have 
\begin{Prop}[see {\cite[Section 5.5]{B2}}]
  \label{pr:order}
  If $m$ is a moment function of order $s\in\RR$
  then there exist constants $c,C>0$ such that
    \begin{gather*}
     c^n\Gamma_s(n)\leq m(n) \leq C^n\Gamma_s(n) \quad \textrm{for every} \quad n\in\NN.
    \end{gather*}
  \end{Prop}

\section{Moment Borel transforms, Gevrey order and Borel summability}
We use the moment function to define moment Borel transforms, the Gevrey order and the Borel summability. We first introduce
\begin{Df}
 \label{df:moment_Borel}
 Let $\kappa\in\NN$ and $m$ be a moment function. Then the linear operator $\Bo_{m,x^{1/\kappa}}\colon \EE[[x^{\frac{1}{\kappa}}]]\to\EE[[x^{\frac{1}{\kappa}}]]$ defined by
 \[
  \Bo_{m,x^{1/\kappa}}\big(\sum_{j=0}^{\infty}u_jx^{j/\kappa}\big):=
  \sum_{j=0}^{\infty}\frac{u_j}{m(j/\kappa)}x^{j/\kappa}
 \]
 is called an \emph{$m$-moment Borel transform with respect to $x^{1/\kappa}$}.
\end{Df}

We define the Gevrey order of formal power series as follows
   \begin{Df}
    \label{df:summab}
    Let $\kappa\in\NN$ and $s\in\RR$. Then
    $\widehat{u}\in\EE[[x^{\frac{1}{\kappa}}]]$ is called a \emph{formal power series of Gevrey order $s$} if
    there exists a disc $D\subset\CC$ with centre at the origin such that
    $\Bo_{\Gamma_s,x^{1/\kappa}}\widehat{u}\in\Oo_{1/\kappa}(D,\EE)$. The space of formal power series of Gevrey 
    order $s$ is denoted by $\EE[[x^{\frac{1}{\kappa}}]]_s$.
    \par
    Analogously, if $\kappa_1,\kappa_2\in\NN$ and $s_1,s_2\in\RR$ then
    $\widehat{u}\in\EE[[t^{\frac{1}{\kappa_1}},z^{\frac{1}{\kappa_2}}]]$ is called a \emph{formal power series of Gevrey order $(s_1,s_2)$} if there exists a disc $D^2\subset\CC^2$ with centre at the origin such that
    $\Bo_{\Gamma_{s_1},t^{1/\kappa_1}}\Bo_{\Gamma_{s_2},z^{1/\kappa_2}}\widehat{u}\in
    \Oo_{1/\kappa_1,1/\kappa_2}(D^2,\EE)$. The space of formal power series of Gevrey 
    order $(s_1,s_2)$ is denoted by $\EE[[t^{\frac{1}{\kappa_1}},z^{\frac{1}{\kappa_2}}]]_{s_1,s_2}$.
   \end{Df}

\begin{Rem}
 \label{re:moment}
 By Proposition \ref{pr:order}, we may replace $\Gamma_s$ (resp. $\Gamma_{s_1}$ and $\Gamma_{s_2}$) in Definition
 \ref{df:summab} by any moment function $m$ of order $s$ (resp. by any moment functions $m_1$ and $m_2$ of orders
 $s_1$ and $s_2$).
\end{Rem}

\begin{Rem}
 \label{re:entire}
 If $\widehat{u}\in\EE[[x^{\frac{1}{\kappa}}]]_s$ and $s\leq 0$ then the formal series $\widehat{u}$ is convergent,
 so its sum $u$ is well defined.
 Moreover, $\widehat{u}\in\EE[[x^{\frac{1}{\kappa}}]]_0 \Longleftrightarrow u\in\Oo_{1/\kappa}(D,\EE)$ and
 $\widehat{u}\in\EE[[x^{\frac{1}{\kappa}}]]_s \Longleftrightarrow u\in\Oo^{-1/s}_{1/\kappa}(\CC,\EE)$ for $s<0$.
\end{Rem}

By Definitions \ref{df:moment_Borel} and \ref{df:summab} we obtain
\begin{Prop}
 \label{pr:properties}
 For every $\widehat{u}\in\EE[[x^{\frac{1}{\kappa}}]]$ the following properties of moment Borel transforms are satisfied:
 \begin{itemize}
  \item $\Bo_{m_1,x^{1/\kappa}}\Bo_{m_2,x^{1/\kappa}}\widehat{u} = \Bo_{m_1m_2,x^{1/\kappa}}\widehat{u}$ for every moment functions $m_1$ and
  $m_2$.
  \item $\Bo_{m,x^{1/\kappa}}\Bo_{1/m,x^{1/\kappa}}\widehat{u}=
  \Bo_{1/m,x^{1/\kappa}}\Bo_{m,x^{1/\kappa}}\widehat{u}=\Bo_{1,x^{1/\kappa}}\widehat{u}=\widehat{u}$ for every moment function $m$.
  \item $\widehat{u}\in\EE[[x^{\frac{1}{\kappa}}]]_{s_1} \Longleftrightarrow \Bo_{m,x^{1/\kappa}}\widehat{u}\in\EE[[x^{\frac{1}{\kappa}}]]_{s_1-s}$  for every $s,s_1\in\RR$ and for every moment function $m$ of order $s$.
 \end{itemize}
\end{Prop}

As a Banach space $\EE$ we will take the space of complex numbers $\CC$ or the space of Gevrey series 
$G_{s,1/\kappa}(r)$ defined below.

\begin{Df}
\label{df:G_s}
Fix $\kappa\in\NN$, $r>0$ and $s\in\RR$. By $G_{s,1/\kappa}(r)$ we denote a Banach space of Gevrey series
\[
 G_{s,1/\kappa}(r):=\{\widehat{\varphi}\in\CC[[z^{\frac{1}{\kappa}}]]_s\colon \Bo_{\Gamma_s,z^{1/\kappa}}\widehat{\varphi}\in\Oo_{1/\kappa}(D_r)\cap C(\overline{D_r})\}
\]
equipped with the norm
\[
 \|\widehat{\varphi}\|_{G_{s,1/\kappa}(r)}:=\max_{|z|\leq r}|\Bo_{\Gamma_s,z^{1/\kappa}}\widehat{\varphi}(z)|.
\]
We also set $G_{s,1/\kappa}:=\varinjlim\limits_{r>0}G_{s,1/\kappa}(r)$. Analogously, we define
$\Oo_{1/\widetilde{\kappa}}(G,G_{s,1/\kappa}):=\varinjlim\limits_{r>0}\Oo_{1/\widetilde{\kappa}}(G,G_{s,1/\kappa}(r))$ and $\Oo^K_{1/\widetilde{\kappa}}(G,G_{s,1/\kappa}):=\varinjlim\limits_{r>0}\Oo^K_{1/\widetilde{\kappa}}(G,G_{s,1/\kappa}(r))$.
\par
Moreover, we denote by $G_{s_2,1/\kappa}[[t]]_{s_1}$ the space of
formal power series $\widehat{u}(t,z)=\sum_{j=0}^{\infty}\widehat{u}_j(z)t^j$ of Gevrey order $s_1$ with
coefficients $\widehat{u}_j(z)\in G_{s_2,1/\kappa}$.
\end{Df}

By Definitions \ref{df:summab}, \ref{df:G_s}, Remark \ref{re:moment} and Proposition 
\ref{pr:properties} we conclude
\begin{Prop}
\label{pr:prop2}
 For every $\kappa\in\NN$, $s, \overline{s}\in\RR$ (resp. $s_1,s_2,\overline{s}\in\RR$) and for every moment function $m$ of order $\overline{s}$ the following conditions are equivalent:
 \begin{itemize}
  \item $\widehat{u}\in\CC[[x^{\frac{1}{\kappa}}]]_s$
  (resp. $\widehat{u}\in\CC[[t,z^{\frac{1}{\kappa}}]]_{s_1,s_2}$),
  \item $\Bo_{\Gamma_s,x^{1/\kappa}}\widehat{u}\in\Oo_{1/\kappa}(D)$ (resp. $\Bo_{\Gamma_{s_1},t}
  \Bo_{\Gamma_{s_2},z^{1/\kappa}}\widehat{u}\in\Oo_{1,1/\kappa}(D^2)$),
  \item there exists $r>0$ such that $\widehat{u}\in G_{s,1/\kappa}(r)$ (resp.
  $\widehat{u}\in G_{s_2, 1/\kappa}(r)[[t]]_{s_1}$),
  \item $\widehat{u}\in G_{s,1/\kappa}$ (resp.
  $\widehat{u}\in G_{s_2,1/\kappa}[[t]]_{s_1}$),
  \item $\Bo_{m,x^{1/\kappa}}\widehat{u}\in\CC[[x^{\frac{1}{\kappa}}]]_{s-\overline{s}}$
   (resp. $\Bo_{m,z^{1/\kappa}}\widehat{u}\in G_{s_2-\overline{s},1/\kappa}[[t]]_{s_1}$). 
 \end{itemize}
\end{Prop}

Now we are ready to define the summability of formal power series in one variable (see Balser \cite{B2})
\begin{Df}
\label{df:summable}
Let $\kappa\in\NN$, $K>0$ and $d\in\RR$. Then $\widehat{u}\in\EE[[x^{\frac{1}{\kappa}}]]$ is called
\emph{$K$-summable in a direction $d$} if there exists a disc-sector $\hat{S}_d$ in a direction $d$ such that
$\Bo_{\Gamma_{1/K},x^{1/\kappa}}\widehat{u}\in\Oo^K_{1/\kappa}(\hat{S}_d,\EE)$.
\end{Df}

\begin{Rem}
\label{re:summable}
 By Definitions \ref{df:G_s} and \ref{df:summable}, $\widehat{u}\in G_{s,1/\kappa}[[t]]$ is $K$-summable in a direction $d$
 if and only if $\Bo_{\Gamma_{1/K},t}\Bo_{\Gamma_{s},z^{1/\kappa}}\widehat{u}\in\Oo^K_{1,1/\kappa}(\hat{S}_d\times D)$.
\end{Rem}

We can now define the multisummability in a multidirection.
\begin{Df}
 Let $k_1>\cdots>k_n>0$. We say that a real vector $(d_1,\dots,d_n)\in\RR^n$ is an
 \emph{admissible multidirection} if
 \begin{gather*}
  |d_j-d_{j-1}| \leq \pi(1/k_j - 1/k_{j-1})/2 \quad \textrm{for} \quad j=2,\dots,n.
 \end{gather*}
 \par
 Let $\mathbf{k}=(k_1,\dots,k_n)\in\RR^n_+$ and  let $\mathbf{d}=(d_1,\dots,d_n)\in\RR^n$ be an
 admissible multidirection.
 We say that a formal power series
 $\widehat{u}\in\EE[[x]]$ is {\em $\mathbf{k}$-multisummable in the
 multidirection $\mathbf{d}$}
 if $\widehat{u}=\widehat{u}_1+\cdots+\widehat{u}_n$, where $\widehat{u}_j\in\EE[[x]]$ is
 $K_j$-summable
 in the direction $d_j$ for $j=1,\dots,n$.
\end{Df}

Following Sanz \cite{S} we extend the notion of summability to two variables
\begin{Df}
\label{df:summable2}
For $\kappa_1,\kappa_2\in\NN$, $K_1,K_2>0$ and $d_1,d_2\in\RR$ the formal power series
$\widehat{u}\in\CC[[t^{\frac{1}{\kappa_1}},z^{\frac{1}{\kappa_2}}]]$ is called
\emph{$(K_1,K_2)$-summable in the direction $(d_1,d_2)$} if there exist disc-sectors $\hat{S}_{d_1}$
and $\hat{S}_{d_2}$ such that
$\Bo_{\Gamma_{1/K_1},t^{1/\kappa_1}}\Bo_{\Gamma_{1/K_2},z^{1/\kappa_2}}\widehat{u}\in
\Oo_{1/\kappa_1,1/\kappa_2}^{K_1,K_2}(\hat{S}_{d_1}\times\hat{S}_{d_2})$.
\end{Df}

\begin{Rem}
By the general theory of moment summability (see \cite[Section 6.5 and Theorem 38]{B2}), we may replace
 $\Gamma_{1/K}$  in Definition \ref{df:summable} (resp. $\Gamma_{1/K_1}$ and $\Gamma_{1/K_2}$ 
 in Definition \ref{df:summable2}) by any 
 moment function $m$ of order $1/K$ (resp. by any moment functions $m_1$ of order $1/K_1$ and $m_2$ of order $1/K_2$).
\end{Rem}
 
 Using moment functions of order $0$ we can formulate the above remark as
 \begin{Prop}
  \label{pr:sum}
  We assume that $\kappa_1,\kappa_2\in\NN$, $K_1,K_2>0$, $d_1,d_2\in\RR$, $\overline{m}_1,\overline{m}_2$ are moment functions of order $0$,
$\widehat{u}\in\CC[[t^{\frac{1}{\kappa_1}},z^{\frac{1}{\kappa_2}}]]$ and $v\in\Oo_{1/\kappa_1,1/\kappa_2}(D^2)$. Then
 \begin{itemize}
  \item $\widehat{u}$ is $K_1$-summable in the direction $d_1$ (with respect to $t^{\frac{1}{\kappa_1}}$) if and only if 
        $\Bo_{\overline{m}_1,t^{1/\kappa_1}}\Bo_{\overline{m}_2,z^{1/\kappa_2}}\widehat{u}$ is
        $K_1$-summable in the direction $d_1$ (with respect to $t^{\frac{1}{\kappa_1}}$).
  \item $\widehat{u}$ is $(K_1,K_2)$-summable in the direction $(d_1,d_2)$ if and only if 
        $\Bo_{\overline{m}_1,t^{1/\kappa_1}}\Bo_{\overline{m}_2,z^{1/\kappa_2}}\widehat{u}$ is
        $(K_1,K_2)$-summable in the direction $(d_1,d_2)$.
  \item $v\in\Oo_{1/\kappa_1,1/\kappa_2}^{K_1}(\hat{S}_{d_1}\times D)$ if and only if
        $\Bo_{\overline{m}_1,t^{1/\kappa_1}}\Bo_{\overline{m}_2,z^{1/\kappa_2}}v\in
        \Oo_{1/\kappa_1,1/\kappa_2}^{K_1}(\hat{S}_{d_1}\times D)$.
  \item $v\in\Oo_{1/\kappa_1,1/\kappa_2}^{K_1,K_2}(\hat{S}_{d_1}\times \hat{S}_{d_2})$ if and only if
        $\Bo_{\overline{m}_1,t^{1/\kappa_1}}\Bo_{\overline{m}_2,z^{1/\kappa_2}}v\in
        \Oo_{1/\kappa_1,1/\kappa_2}^{K_1,K_2}(\hat{S}_{d_1}\times \hat{S}_{d_2})$.   
 \end{itemize}

 \end{Prop}

\section{Moment operators}
     In this section we recall the notion of moment differential operators constructed by Balser and Yoshino
     \cite{B-Y} and the concept of moment pseudodifferential operators introduced in our previous
     papers \cite{Mic7,Mic8}.
     
   \begin{Df}
    Let $m$ be a moment function. Then the linear operator
    $\partial_{m_,x}\colon\EE[[x]]\to\EE[[x]]$
    defined by
    \[
     \partial_{m,x}\Big(\sum_{j=0}^{\infty}\frac{u_{j}}{m(j)}x^{j}\Big):=
     \sum_{j=0}^{\infty}\frac{u_{j+1}}{m(j)}x^{j}
    \]
    is called the \emph{$m$-moment differential operator $\partial_{m,x}$}.
    \par
    More generally, if $\kappa\in\NN$ then the linear operator 
   $\partial_{m_,x^{1/\kappa}}\colon\EE[[x^{\frac{1}{\kappa}}]]\to\EE[[x^{\frac{1}{\kappa}}]]$
    defined by
    \[
     \partial_{m,x^{1/\kappa}}\Big(\sum_{j=0}^{\infty}\frac{u_{j}}{m(j/\kappa)}x^{j/\kappa}\Big):=
     \sum_{j=0}^{\infty}\frac{u_{j+1}}{m(j/\kappa)}x^{j/\kappa}
    \]
    is called the \emph{$m$-moment $1/\kappa$-fractional differential operator $\partial_{m,x^{1/\kappa}}$}.
    \par
    Moreover, the right-inversion operator 
    $\partial^{-1}_{m,x^{1/\kappa}}\colon\EE[[x^{\frac{1}{\kappa}}]]\to\EE[[x^{\frac{1}{\kappa}}]]$
    given by
    $$
     \partial^{-1}_{m,x^{1/\kappa}}\Big(\sum_{j=0}^{\infty}\frac{u_{j}}{m(j/\kappa)}x^{j/\kappa}\Big):=
     \sum_{j=1}^{\infty}\frac{u_{j-1}}{m(j/\kappa)}x^{j/\kappa}$$
    is called the \emph{$m$-moment $1/\kappa$-fractional integration operator $\partial^{-1}_{m,x^{1/\kappa}}$}.
    \end{Df}
    
     Below we present most important examples of moment differential operators. Other examples, including also integro-differential operators, can be found in
     \cite[Example 3]{Mic8}.
     \begin{Ex} 
    \label{ex:operator}
     If $m(u)=\Gamma_1(u)$ then the operator $\partial_{m,x}$ coincides with the usual differentiation $\partial_x$.
     More generally, if $s>0$ and $m(u)=\Gamma_s(u)$ then the operator $\partial_{m,x}$ satisfies
     \[
      (\partial_{m,x}\widehat{u})(x^s)=\partial^s_x(\widehat{u}(x^s)),
     \]
     where $\partial^s_x$ denotes the Caputo fractional derivative of order $s$
     defined by
       $$
     \partial^{s}_{x}\Big(\sum_{j=0}^{\infty}\frac{u_{j}}{\Gamma_s(j)}x^{sj}\Big):=
     \sum_{j=0}^{\infty}\frac{u_{j+1}}{\Gamma_s(j)}x^{sj}.$$
     \end{Ex}
  
 The moment differential operator $\partial_{m,z}$ is well-defined for every $\varphi\in\Oo(D)$. In addition,
 we have the following integral representation of $\partial^n_{m,z}\varphi$.
\begin{Prop}[see {\cite[Proposition 3]{Mic7}}]
 Let $\varphi\in \Oo(D_r)$ and $m$ be a moment function of order $1/k>0$. Then for every $|z|<\varepsilon<r$ and 
 $n\in\NN$ we have
 \[
  \partial_{m,z}^n\varphi(z) = \frac{1}{2\pi i} \oint_{|w|=\varepsilon} \varphi(w)
  \int_0^{\infty(\theta)}\zeta^{n}E_m(z\zeta)\frac{e_m(w\zeta)}{w\zeta}\,d\zeta\,dw,
 \]
 where $\theta\in (-\arg w-\frac{\pi}{2k}, -\arg w + \frac{\pi}{2k})$.
\end{Prop}

Using the above formula, we have defined in \cite[Definition 8]{Mic7} a moment
pseudodifferential operator $\lambda(\partial_{m,z})\colon \Oo(D)\to\Oo(D)$ as an operator satisfying
\begin{eqnarray*}
  \lambda(\partial_{m,z}) E_m(\zeta z):=\lambda(\zeta)E_m(\zeta z)&\textrm{for}& |\zeta|\geq r_0.
\end{eqnarray*}
Namely, if $\lambda(\zeta)$ is 
an analytic function for $|\zeta|\geq r_0$ then $\lambda(\partial_{m,z})$ is defined by
\[
\lambda(\partial_{m,z})\varphi(z):=\frac{1}{2\pi i} \oint_{|w|=\varepsilon} \varphi(w)
  \int_{r_0e^{i\theta}}^{\infty(\theta)}\lambda(\zeta)E_m(\zeta z)\frac{e_m(\zeta w)}{\zeta w}\,d\zeta\,dw
\]
for every $\varphi\in\Oo(D_r)$ and $|z|<\varepsilon < r$, where
$\theta\in (-\arg w-\frac{\pi}{2k}, -\arg w + \frac{\pi}{2k})$.
\par
In \cite{Mic8} we have extended this definition to the case where $\lambda(\zeta)$ is an analytic function of the variable
$\xi=\zeta^{1/\kappa}$ for $|\zeta|\geq r_0$ (for some $\kappa\in\NN$ and $r_0>0$).
Since 
$(\partial_{m,z}\varphi)(z^{\kappa})=\partial_{\widetilde{m},z}^{\kappa}(\varphi(z^{\kappa}))$ for every $\varphi\in\Oo(D)$, where $\widetilde{m}(u):=m(u/\kappa)$ (see \cite[Lemma 3]{Mic7}),
the operator $\lambda(\partial_{m,z})$ should satisfy the formula
\begin{eqnarray}
 \label{eq:kappa}
 (\lambda(\partial_{m,z})\varphi)(z^{\kappa})=\lambda(\partial^{\kappa}_{\widetilde{m},z})(\varphi(z^{\kappa})) &
 \textrm{for every} & \varphi\in\Oo_{1/\kappa}(D).
\end{eqnarray}
For this reason we have
\begin{Df}[{\cite[Definition 13]{Mic8}}]
\label{df:pseudo_1}
Let $m$ be a moment function of order $1/k>0$ and
 $\lambda(\zeta)$ be an analytic function of the variable
$\xi=\zeta^{1/\kappa}$ for $|\zeta|\geq r_0$
(for some $\kappa\in\NN$ and $r_0>0$) of polynomial growth at infinity.
A \emph{moment pseudodifferential operator}
 $\lambda(\partial_{m,z})\colon\Oo_{1/\kappa}(D)\to\Oo_{1/\kappa}(D)$ (or, more generally, $\lambda(\partial_{m,z})\colon\EE[[z^{\frac{1}{\kappa}}]]_0\to\EE[[z^{\frac{1}{\kappa}}]]_0$) is defined by
 \begin{equation}
  \label{eq:lambda}
  \lambda(\partial_{m,z})\varphi(z):=\frac{1}{2\kappa\pi i} \oint^{\kappa}_{|w|=\varepsilon}\varphi(w)
  \int_{r_0e^{i\theta}}^{\infty(\theta)}\lambda(\zeta)
  E_{\widetilde{m}}(\zeta^{1/\kappa} z^{1/\kappa})\frac{e_m(\zeta w)}{\zeta w}\,d\zeta\,dw
 \end{equation}
for every $\varphi\in\Oo_{1/\kappa}(D_r)$ and $|z|<\varepsilon < r$, where $\widetilde{m}(u):=m(u/\kappa)$, $E_{\widetilde{m}}(\zeta^{1/\kappa} z^{1/\kappa})=\sum_{n=0}^{\infty}\frac{\zeta^{n/\kappa}z^{n/\kappa}}{\widetilde{m}(n)}$, $\theta\in (-\arg w-\frac{\pi}{2k}, -\arg w + \frac{\pi}{2k})$ and $\oint_{|w|=\varepsilon}^{\kappa}$ means that we integrate $\kappa$ times along the positively oriented circle of radius $\varepsilon$. Here the integration in the inner integral is taken over a ray $\{re^{i\theta}\colon r\geq r_0\}$.
\end{Df}

Observe that
\begin{multline*}
(\lambda(\partial_{m,z})\varphi)(z^{\kappa})=\frac{1}{2\kappa\pi i}\oint_{|w|=\varepsilon}^{\kappa} \varphi(w)
  \int\limits_{r_0e^{i\theta}}^{\infty(\theta)}\lambda(\zeta)E_{\widetilde{m}}(\zeta^{1/\kappa}z)
  \frac{e_m(\zeta w)}{\zeta w}\,d\zeta\,dw\\
  =\frac{1}{2\pi i}\oint_{|w^{\kappa}|=\varepsilon}\varphi(w^{\kappa})
  \int\limits_{r_0^{1/\kappa}e^{i\theta/\kappa}}^{\infty(\theta/\kappa)}\lambda(\zeta^k)
  E_{\widetilde{m}}(\zeta z)\frac{e_{\widetilde{m}}(\zeta w)}{\zeta w}\,d\zeta\,dw=
 \lambda(\partial^{\kappa}_{\widetilde{m},z})(\varphi(z^{\kappa})),
\end{multline*}
so (\ref{eq:kappa}) holds for the operators $\lambda(\partial_{m,z})$ defined by (\ref{eq:lambda}).
\par
Immediately by the definition, we obtain the following connection between the moment Borel transform and the moment
differentiation.
\begin{Prop}
\label{pr:commutation}
Let $m$ and $m'$ be moment functions of positive orders. Then
the operators $\Bo_{m',x}, \partial_{m,x}\colon\EE[[x]]\to\EE[[x]]$ satisfy the following commutation formulas
for every $\widehat{u}\in\EE[[x]]$ and for $\overline{m}=mm'$:
\begin{enumerate}
\item[i)] $\Bo_{m',x}\partial_{m,x}\widehat{u}=\partial_{\overline{m},x}\Bo_{m',x}\widehat{u}$,
\item[ii)] $\Bo_{m',x}P(\partial_{m,x})\widehat{u}=P(\partial_{\overline{m},x})\Bo_{m',x}\widehat{u}$
for any polynomial $P$ with constant coefficients.
\end{enumerate}
\end{Prop}

The same commutation formula holds if we replace $P(\partial_{m,x})$ by $\lambda(\partial_{m,x})$. Namely, we have
\begin{Prop}[see {\cite[Proposition 8]{Mic8}}]
\label{pr:commutation2}
Let $m$ and $m'$ be moment functions of positive orders and $\lambda(\zeta)$ be an analytic function of the variable
$\xi=\zeta^{1/\kappa}$ for $|\zeta|\geq r_0$
(for some $\kappa\in\NN$ and $r_0>0$) of polynomial growth at infinity. Then
the operators $\Bo_{m',x^{1/\kappa}}, \lambda(\partial_{m,x})\colon\EE[[x^{\frac{1}{\kappa}}]]_0\to\EE[[x^{\frac{1}{\kappa}}]]_0$ 
satisfy the commutation formula
$$\Bo_{m',x^{1/\kappa}}\lambda(\partial_{m,x})\widehat{u}=\lambda(\partial_{\overline{m},x})
\Bo_{m',x^{1/\kappa}}\widehat{u}$$
for every $\widehat{u}\in\EE[[x^{\frac{1}{\kappa}}]]_0$ and for $\overline{m}=mm'$.
\end{Prop}

Using Proposition \ref{pr:commutation2} we are able to extend Definition \ref{df:pseudo_1} to the formal power series and to the moment functions of real orders. 
\begin{Df}[{\cite[Definition 14]{Mic8}}]
\label{df:pseudo}
 Let $s\in\RR$, $m$ be a moment function of order $\widetilde{s}\in\RR$ and 
 $\lambda(\zeta)$ be an analytic function of the variable
$\xi=\zeta^{1/\kappa}$ for $|\zeta|\geq r_0$
of polynomial growth at infinity.
A \emph{moment pseudodifferential operator
 $\lambda(\partial_{m,z})$} for the formal power series $\widehat{\varphi}\in\EE[[z^{\frac{1}{\kappa}}]]_s$
 is defined by
 \begin{equation}
  \label{eq:l}
  \lambda(\partial_{m,z})\widehat{\varphi}(z):=\Bo_{\Gamma_{-\overline{s}},z^{1/\kappa}}\lambda(\partial_{\overline{m},z})
  \Bo_{\Gamma_{\overline{s}},z^{1/\kappa}}\widehat{\varphi}(z),
 \end{equation}
 where $\overline{m}:=m\Gamma_{\overline{s}}$, $\overline{s}:=\max\{s,1-\widetilde{s}\}$
 and the operator
 $\lambda(\partial_{\overline{m},z})$ is constructed in Definition \ref{df:pseudo_1}.
\end{Df}

\begin{Rem}
 Note that the choice of $\overline{s}:=\max\{s,1-\widetilde{s}\}$ guarantees that 
 $\Bo_{\Gamma_{\overline{s}},z^{1/\kappa}}\widehat{\varphi}(z)\in\EE[[z^{\frac{1}{\kappa}}]]_0$ and
 that $\overline{m}$ is a moment function of positive order $\overline{s}+\widetilde{s}\geq 1$. Thanks to that, the operator
 $\lambda(\partial_{m,z})$ given by (\ref{eq:l}) is well-defined. 
\end{Rem}

\begin{Df}[{\cite[Definition 9]{Mic7}}]
   We define a \emph{pole order $q\in\QQ$} and a \emph{leading term $\lambda_0\in\CC\setminus\{0\}$}
   of $\lambda(\zeta)$ as the numbers satisfying the formula
   $\lim_{\zeta\to\infty}\lambda(\zeta)/\zeta^q=\lambda_0$.
   We write it also $\lambda(\zeta)\sim\lambda_0\zeta^q$.
\end{Df}

At the end of the section we recall the estimate given in \cite{Mic8}
\begin{Lem}[{\cite[Lemma 2]{Mic8}}]
\label{le:estimation}
Let $\widehat{\varphi}\in\CC[[z^{\frac{1}{\kappa}}]]_s$, $s\leq 0$, $m$ be a moment function of order $1/k>0$ and $\lambda(\partial_{m,z})$
be a moment pseudodifferential
operator with $\lambda(\zeta)\sim\lambda_0\zeta^q$ and $q\in\QQ$.
Then there exist $r>0$ and $A,B<\infty$ such that
\begin{gather*}
 \sup_{|z|<r}|\lambda^j(\partial_{m,z})\varphi(z)|\leq |\lambda_0|^jAB^{j}\Gamma_{q^+(s+1/k)}(j)\quad \textrm{for} 
 \quad j=0,1,\dots,
\end{gather*}
where $q^+:=\max\{0,q\}$.
\end{Lem}

\section{Formal solutions and Gevrey estimates}
   In a similar way to \cite{Mic7} we generalise the results of \cite{Mic8} to inhomogeneous case. Without loss
of generality we may assume that the initial data vanish. So it is sufficient to consider
the following Cauchy problem for general inhomogeneous linear moment partial differential equation with
 constant coefficients
 \begin{equation}
  \label{eq:inhomo0}
  \left\{
   \begin{array}{lll}
    P(\partial_{m_1,t},\partial_{m_2,z})u=\widehat{f}\in\CC[[t,z]]_{\tilde{s}_1,\tilde{s}_2}&&\\
     \partial_{m_1,t}^j u(0,z)=0&
    \textrm{for}&j=0,\dots,n-1,
   \end{array}
  \right.
 \end{equation}
 where $m_1$, $m_2$ are moment functions of orders $s_1,s_2\in\RR$ respectively,
 $(\tilde{s}_1,\tilde{s}_2)\in\RR^2$ is a Gevrey order of the inhomogeneity $\widehat{f}$ and
 $P(\lambda,\zeta)$ is a general polynomial of two variables, which is of order $n$ with respect to $\lambda$.
 It means that 
 \begin{equation*}
     P(\lambda,\zeta)=P_0(\zeta)\lambda^n-\sum_{j=1}^nP_j(\zeta)\lambda^{n-j}
     =P_0(\zeta)\widetilde{P}(\lambda,\zeta)=P_0(\zeta)\prod_{\alpha=1}^l(\lambda-\lambda_{\alpha}(\zeta))^{n_{\alpha}},
 \end{equation*}
 where $\lambda_1(\zeta),\dots,\lambda_l(\zeta)$ are the roots of the characteristic equation
   $P(\lambda,\zeta)=0$ with multiplicities $n_1,\dots,n_l$ ($n_1+\cdots+n_l=n$) respectively.
     Since $\lambda_{\alpha}(\zeta)$ are algebraic functions,
     we may assume that there exist $\kappa\in\NN$ and $r_0<\infty$ such that
   $\lambda_{\alpha}(\zeta)$ are holomorphic functions of the variable $\xi=\zeta^{1/\kappa}$
   (for $|\zeta|\geq r_0$) and, moreover, there exist $\lambda_{\alpha}\in\CC\setminus\{0\}$ and 
   $q_{\alpha}=\mu_{\alpha}/\nu_{\alpha}$
   (for some relatively prime numbers $\mu_{\alpha}\in\ZZ$ and $\nu_{\alpha}\in\NN$) such that
   $\lambda_{\alpha}(\zeta)\sim\lambda_{\alpha}\zeta^{q_{\alpha}}$ for $\alpha=1,\dots,l$.
    
   Hence $\lambda_{\alpha}(\partial_{m_2,z})$ are well-defined moment pseudodifferential operators
   and consequently also the operator
   $$
    \widetilde{P}(\partial_{m_1,t},\partial_{m_2,z})=(\partial_{m_1,t}-\lambda_1(\partial_{m_2,z}))^{n_1}\cdots
    (\partial_{m_1,t}-\lambda_l(\partial_{m_2,z}))^{n_l}
   $$
   is well-defined.
 
 If $P_0(\zeta)\neq \textrm{const.}$ then the formal solution of (\ref{eq:inhomo0}) is not uniquely determined.
 For this reason we choose a formal power series $\widehat{g}\in\CC[[t,z]]_{\tilde{s}_1,\tilde{s}_2}$
 satisfying the equation $P_0(\partial_{m_2,z})\widehat{g}=\widehat{f}$. For such $\widehat{g}$ we 
 may construct the uniquely determined solution $\widehat{u}$ of
 \begin{equation*}
  \left\{
   \begin{array}{lll}
    \widetilde{P}(\partial_{m_1,t},\partial_{m_2,z})u=\widehat{g}\in\CC[[t,z]]_{\tilde{s}_1,\tilde{s}_2}&&\\
     \partial_{m_1,t}^j u(0,z)=0&
    \textrm{for}&j=0,\dots,n-1,
   \end{array}
  \right.
 \end{equation*}
 which is also a formal solution of (\ref{eq:inhomo0}) and is called the
 \emph{formal solution of (\ref{eq:inhomo0}) determined by $\widehat{g}$}.

In a similar way to \cite[Theorem 1]{Mic8}, we generalise the results for the analytic Cauchy data given in 
\cite[Theorems 7 and 8]{Mic7}
  as follows
\begin{Th}
  \label{th:gevrey}
  Let $\widehat{u}$ be a formal solution of (\ref{eq:inhomo0}) determined by
  $\widehat{g}\in\CC[[t,z]]_{\tilde{s}_1,\tilde{s}_2}$.
  Then $\widehat{u}=\sum_{\alpha=1}^l\sum_{\beta=1}^{n_{\alpha}}\widehat{u}_{\alpha\beta}$ with
   $\widehat{u}_{\alpha\beta}$ being a formal solution of simple inhomogeneous pseudodifferential equation
   \begin{equation}
   \label{eq:gevrey}
    \left\{
    \begin{array}{l}
     (\partial_{m_1,t}-\lambda_{\alpha}(\partial_{m_2,z}))^{\beta} u_{\alpha\beta}=\widehat{g}_{\alpha\beta}\\
     \partial_{m_1,t}^j u_{\alpha\beta}(0,z)=0\ \textrm{for}\ j=0,\dots,\beta-1
    \end{array}
    \right.
   \end{equation}
   where $\widehat{g}_{\alpha\beta}(t,z):=d_{\alpha\beta}(\partial_{m_2,z})\widehat{g}(t,z)\in
   \CC[[t,z^{\frac{1}{\kappa}}]]_{\tilde{s}_1,\tilde{s}_2}$ and $d_{\alpha\beta}(\zeta)$ is a holomorphic function
   of the variable $\xi=\zeta^{\frac{1}{\kappa}}$ and of polynomial growth.
   \par
   Moreover, if $q_{\alpha}$ is a pole order of $\lambda_{\alpha}(\zeta)$ for some $\alpha\in\{1,\dots,l\}$
   and $q_{\alpha}^+:=\max\{0,q_{\alpha}\}$,
   then a formal solution $\widehat{u}_{\alpha\beta}$ is a Gevrey series of order
   $Q_{\alpha}:=\max\{q_{\alpha}^+(s_2+\tilde{s}_2) - s_1, \tilde{s}_1\}$
   with respect to $t$. More precisely,
   $\widehat{u}_{\alpha\beta}\in\CC[[t,z^{\frac{1}{\kappa}}]]_{Q_{\alpha},\tilde{s}_2}$
   or, equivalently, $\widehat{u}_{\alpha\beta}\in G_{\tilde{s}_2,1/\kappa}[[t]]_{Q_{\alpha}}$.
\end{Th}
\begin{proof}
  For fixed $\overline{s}_2>\max\{\tilde{s}_2,-s_2\}$ we define $\widehat{v}:=\Bo_{\Gamma_{\tilde{s}_1},t}\Bo_{\Gamma_{\overline{s}_2},z}\widehat{u}$.
 By Proposition \ref{pr:commutation}, $\widehat{v}$ is a formal solution of
\[
    \left\{
    \begin{array}{l}
     P(\partial_{\tilde{m}_1,t},\partial_{\overline{m}_2,z})v=\Bo_{\Gamma_{\tilde{s}_1},t}\Bo_{\Gamma_{\overline{s}_2},z}\widehat{f}\\
     \partial^j_{\tilde{m}_1,t} v(0,z)=0\ \textrm{for}\ j=0,\dots,n-1,
    \end{array}
    \right.
\]
determined by $\widehat{h}=\Bo_{\Gamma_{\tilde{s}_1},t}\Bo_{\Gamma_{\overline{s}_2},z}\widehat{g}$,
where $\tilde{m}_1:=m_1\Gamma_{\tilde{s}_1}$ and $\overline{m}_2:=m_2\Gamma_{\overline{s}_2}$.
Since $\overline{m}_2$ is a moment function of order $\overline{s}_2+s_2>0$ and 
$\widehat{h}\in\CC[[t,z]]_{0,\tilde{s}_2-\overline{s}_2}$
with $\tilde{s}_2-\overline{s}_2<0$, repeating the proof of
\cite[Theorem 7]{Mic7} we conclude that
 $\widehat{v}=\sum_{\alpha=1}^l\sum_{\beta=1}^{n_{\alpha}}\widehat{v}_{\alpha\beta}$ with $\widehat{v}_{\alpha\beta}$ being
 a formal solution of
\[
    \left\{
    \begin{array}{l}
     (\partial_{\tilde{m}_1,t}-\lambda_{\alpha}(\partial_{\overline{m}_2,z}))^{\beta} v_{\alpha\beta}=\widehat{h}_{\alpha\beta}(t,z)\\
     \partial_{\tilde{m}_1,t}^j v_{\alpha\beta}(0,z)=0\ \textrm{for}\ j=0,\dots,\beta-1,
    \end{array}
    \right.
\]
where $\widehat{h}_{\alpha\beta}(t,z)=d_{\alpha\beta}(\partial_{\overline{m}_2,z})\widehat{h}(t,z)\in
\CC[[t,z^{\frac{1}{\kappa}}]]_{0,\tilde{s}_2-\overline{s}_2}$ and $d_{\alpha\beta}(\zeta)$ is a 
holomorphic function of the variable $\xi=\zeta^{1/\kappa}$ and of polynomial growth.
Hence, by Definition \ref{df:pseudo},  
$\widehat{u}=\sum_{\alpha=1}^l\sum_{\beta=1}^{n_{\alpha}}\widehat{u}_{\alpha\beta}$
where $\widehat{u}_{\alpha\beta}=\Bo_{\Gamma_{-\tilde{s}_1},t}\Bo_{\Gamma_{-\overline{s}_2},z^{1/\kappa}}\widehat{v}_{\alpha\beta}$
satisfies (\ref{eq:gevrey})
with
\begin{eqnarray*}
 \widehat{g}_{\alpha\beta}(t,z)&=&\Bo_{\Gamma_{-\tilde{s}_1},t}\Bo_{\Gamma_{-\overline{s}_2},z^{1/\kappa}}\widehat{h}_{\alpha\beta}(t,z)=
 \Bo_{\Gamma_{-\tilde{s}_1},t}\Bo_{\Gamma_{-\overline{s}_2},z^{1/\kappa}}d_{\alpha\beta}(\partial_{\overline{m}_2,z})
\widehat{h}(t,z)\\
&=&d_{\alpha\beta}(\partial_{m_2,z})\Bo_{\Gamma_{-\tilde{s}_1},t}\Bo_{\Gamma_{-\overline{s}_2},z}\widehat{h}(t,z) =
d_{\alpha\beta}(\partial_{m_2,z})\widehat{g}(t,z)
\end{eqnarray*}
for $\beta=1,\dots,n_{\alpha}$ and $\alpha=1,\dots,l$. 
\par
To find the Gevrey order of $\widehat{v}_{\alpha\beta}(t,z)=\sum_{j=0}^{\infty}\frac{v_{\alpha\beta j}(z)}{\tilde{m}_1(j)}t^j$ with respect to $t$,
observe that
by \cite[Lemma 8]{Mic7}
\begin{equation}
 \label{eq:v}
 \widehat{v}_{\alpha\beta}(t,z)=\sum_{k=\beta-1}^{\infty}{k\choose \beta-1}(\partial^{-1}_{\tilde{m}_1,t})^{k+1} 
 \lambda_{\alpha\beta}^{k-\beta+1}(\partial_{\overline{m}_2,z})\widehat{h}_{\alpha\beta}(t,z).
\end{equation}
Since
\begin{equation}
\label{eq:h}
\widehat{h}_{\alpha\beta}(t,z)=\sum_{j=0}^{\infty}\frac{h_{\alpha\beta, j}(z)}{\tilde{m}_1(j)}t^j
\end{equation}
is a Gevrey series of order $0$ with respect to $t$ and $\tilde{m}_1$ is a~moment function of order $s_1+\tilde{s}_1$, there exist $r>0$ and
$A,B<\infty$ such that
\begin{gather*}
 \sup_{|z|<r}|h_{\alpha\beta, j}(z)|\leq AB^j\Gamma_{s_1+\tilde{s}_1}(j) \quad \textrm{for} \quad j=0,1,\dots
\end{gather*}
Substituting (\ref{eq:h}) into (\ref{eq:v}) we obtain
\begin{gather*}
 \widehat{v}_{\alpha\beta}(t,z)=\sum_{j=\beta}^{\infty}\frac{t^j}{\tilde{m}_1(j)}\sum_{k=\beta-1}^{j-1}{k\choose \beta-1}
 \lambda_{\alpha\beta}^{k-\beta+1}(\partial_{\overline{m}_2,z})h_{\alpha\beta, j-k-1}(z).
\end{gather*}
Hence, by Lemma \ref{le:estimation}, there exist $r>0$ and $C,D<\infty$ such that for every $z\in D_r$
\begin{align*}
|v_{\alpha\beta j}(z)|&\leq
\sum_{k=\beta-1}^{j-1}{k\choose \beta-1}
| \lambda_{\alpha\beta}^{k-\beta+1}(\partial_{\overline{m}_2,z})h_{\alpha\beta, j-k-1}(z)|\\
&\leq CD^j \sum_{k=\beta-1}^{j-1}\Gamma_{q_{\alpha}^+(\tilde{s}_2-\overline{s}_2+s_2+\overline{s}_2)}(k)
\Gamma_{s_1+\tilde{s}_1}(j-k-1).
\end{align*}
Hence there exist $A,B<\infty$ such that
\begin{gather*}
 \sup_{|z|<r}|v_{\alpha\beta j}(z)|\leq AB^j\Gamma_{Q_{\alpha}+s_1}(j)\quad\textrm{for}\quad j=0,1,\dots,
\end{gather*}
where $Q_{\alpha}=\max\{q_{\alpha}^+(s_2+\tilde{s}_2) - s_1, \tilde{s}_1\}$.

It means that $\widehat{v}_{\alpha\beta}(t,z)=\sum_{j=0}^{\infty}\frac{v_{\alpha\beta j}(z)}{\tilde{m}_1(j)}t^j\in
G_{\tilde{s}_2-\overline{s}_2,1/\kappa}[[t]]_{Q_{\alpha}-\tilde{s}_1}$.
Finally, by Proposition \ref{pr:prop2} we conclude that
$
\widehat{u}_{\alpha\beta}=\Bo_{\Gamma_{-\tilde{s}_1},t}\Bo_{\Gamma_{-\overline{s}_2},z^{1/\kappa}}\widehat{v}_{\alpha\beta}\in
G_{\tilde{s}_2,1/\kappa}[[t]]_{Q_{\alpha}}
$
or, equivalently, $\widehat{u}_{\alpha\beta}\in\CC[[t,z^{\frac{1}{\kappa}}]]_{Q_{\alpha},\tilde{s}_2}$.
\end{proof}

\section{Integral representation and analytic solution}
    By Proposition \ref{pr:sum}, 
    to study analytic continuation or $K$-summability of $\widehat{u}$ satisfying (\ref{eq:gevrey}) we may 
    replace $m_1$ and $m_2$
    by $\Gamma_{s_1}$ and $\Gamma_{s_2}$, where $s_1$ and $s_2$ are orders of $m_1$ and $m_2$ respectively.
    So by Propositions \ref{pr:commutation} and \ref{pr:commutation2}, we may assume that $\widehat{u}$ is a formal solution of
    \begin{equation}
     \label{eq:gamma}
     \left\{
     \begin{array}{l}
     (\partial_{\Gamma_{s_1},t}-\lambda(\partial_{\Gamma_{s_2},z}))^{\beta} u=\widehat{g}\\
      \partial_{\Gamma_{s_1},t}^j u(0,z)=0\ \textrm{for}\ j=0,\dots,\beta-1,
     \end{array}
     \right.
    \end{equation}
    where $\lambda(\zeta)$ is a root of the characteristic equation of (\ref{eq:inhomo0}). It means that $\lambda(\zeta)$ is an analytic
    function of the variable $\xi=\zeta^{1/\kappa}$ for $|\zeta|\geq r_0$ and $\lambda(\zeta)\sim\lambda_0\zeta^q$.
    
   By \cite[Lemma 8]{Mic7} the formal solution of (\ref{eq:gamma}) has the power series representation
   \begin{equation}
    \label{eq:formal}
    \widehat{u}(t,z)=\sum_{j=\beta-1}^{\infty}{j \choose \beta -1}
    (\partial_{\Gamma_{s_1},t}^{-1})^{j+1}\lambda^{j-\beta+1}(\partial_{\Gamma_{s_2},z})
    \widehat{g}(t,z).
   \end{equation}
   To find the integral representation of $\widehat{u}$, we first show
   \begin{Lem}
    If $\widehat{\varphi}(x)\in\EE[[x]]$, $k\in\NN$ and $s>0$ then
   \begin{equation}
   \label{eq:integral}
   (\partial^{-1}_{\Gamma_s,x})^k\widehat{\varphi}(x)=
   -\int_0^{x^{1/s}}\widehat{\varphi}(y^s)\partial_y\frac{(x^{1/s}-y)^{ks}}{\Gamma_s(k)}\,dy.
   \end{equation}
   \end{Lem}
   \begin{proof}
    After the change of variables $t:=y/x^{1/s}$, the right-hand side of (\ref{eq:integral}) is equal to
    $$
    \textrm{RHS}=\frac{x^k}{\Gamma(ks)}\int_0^1(1-t)^{ks-1}\widehat{\varphi}(t^sx)\,dt.
    $$
    Since $\widehat{\varphi}(x)=\sum_{n=0}^{\infty}\frac{\varphi_n}{\Gamma_s(n)}x^n$, we have 
    $$
    \textrm{RHS}=\sum_{n=0}^{\infty}\frac{\varphi_n x^{n+k}}{\Gamma(ks)\Gamma(1+ns)}
    \int_0^1(1-t)^{ks-1}t^{ns}\,dt.
    $$
    Moreover, using the beta integral formula we conclude
    $$
     \int_0^1(1-t)^{ks-1}t^{ns}\,dt=B(ks,1+ns)=\frac{\Gamma(ks)\Gamma(1+ns)}{\Gamma(1+(k+n)s)}.
    $$
    Hence
    $$
     \textrm{RHS} = \sum_{n=0}^{\infty}\frac{\varphi_n}{\Gamma(1+(k+n)s)}x^{n+k}=
     \sum_{n=k}^{\infty}\frac{\varphi_{n-k}}{\Gamma(1+ns)}x^{n}=\textrm{LHS}.
    $$
   \end{proof}

   We also prove
  \begin{Lem}
    \label{le:e}
    Let $s>0$ and $\beta\in\NN$. Then the function 
    \begin{equation}
     \label{eq:e}
     e_{s,\beta}(x):=\sum_{j=\beta}^{\infty}{j -1 \choose \beta -1}\frac{x^{j}}{\Gamma_{s}(j)}
    \end{equation}
    possesses the following properties
    \begin{enumerate}
     \item[(a)] $e_{s,\beta}(x)\in\Oo^{1/s}(\CC)$,
     \item[(b)] if $s<2$ and $\arg x \in (s\pi/2,2\pi-s\pi/2)$ then $e_{s,\beta}(x)\to 0$ as $x\to\infty$.
    \end{enumerate}
  \end{Lem}
  \begin{proof}
   Observe that by \cite[Lemma 6]{B2} the properties (a) and (b) are satisfied by the Mittag-Leffler function
   $\mathbf{E}_s(x)=\sum_{j=0}^{\infty}\frac{x^j}{\Gamma_s(j)}$.
   Moreover, we can express the function $e_{s,\beta}(x)$ in terms of $\mathbf{E}_s(x)$
   as
   $$
   e_{s,\beta}(x)=\frac{1}{(\beta-1)!}x^{\beta}\Big(\frac{\mathbf{E}_s(x)-1}{x}\Big)^{(\beta-1)}.
   $$
   Hence $e_{s,\beta}$ also satisfies (a) and (b).
  \end{proof}

   Using the definition of moment pseudodifferential operators and the power series representation of $\widehat{u}$
   we can find the integral representation of solution $u$
   \begin{Prop}
   \label{pr:integral}
   Let $\lambda(\zeta)\sim\lambda_0\zeta^q$ be an analytic
    function of the variable $\xi=\zeta^{1/\kappa}$ for $|\zeta|\geq r_0$. We also assume that $g\in\Oo_{1,1/\kappa}(D^2)$, $s_1,s_2>0$ and
    $s_1\geq qs_2$. Then the solution $u$ of (\ref{eq:gamma}) belongs to
   the space $\Oo_{1,1/\kappa}(D^2)$ and has the
   integral representation
   \begin{equation*}
     u(t,z)=\frac{-1}{2\kappa\pi i} \int_0^{t^{\frac{1}{s_1}}}\oint_{|w|=\varepsilon}^{\kappa} g(\tau^{s_1},w)\partial_{\tau}k(t,\tau,z,w)
  \,dw\,d\tau,
    \end{equation*}
    where 
    $$
    k(t,\tau,z,w):=\int_{r_0e^{i\theta}}^{\infty(\theta)}\lambda^{-\beta}(\zeta)e_{s_1,\beta}((t^{\frac{1}{s_1}}-\tau)^{s_1}\lambda(\zeta))
    E_{\Gamma_{\frac{s_2}{\kappa}}}(\zeta^{\frac{1}{\kappa}} z^{\frac{1}{\kappa}})\frac{e_{\Gamma_{s_2}}(\zeta w)}{\zeta w}\,d\zeta
    $$
    and  $e_{s_1,\beta}(x)$ is given by (\ref{eq:e}).
    \par
    Moreover, if $s_1>qs_2$ and $g\in\Oo_{1,1/\kappa}^{\frac{1}{s_1-qs_2}}(\CC\times D)$ then also
    $u\in\Oo_{1,1/\kappa}^{\frac{1}{s_1-qs_2}}(\CC\times D)$.
   \end{Prop}
   \begin{proof}
    Since $g\in\Oo_{1,1/\kappa}(D^2)$ and $s_1\geq qs_2$, by Theorem \ref{th:gevrey} we get $u\in\Oo_{1,1/\kappa}(D^2)$.
    To show that the integral given in the definition of
    $k(t,\tau,z,w)$
    is convergent, observe that by Lemma \ref{le:e}, Definitions \ref{df:moment} and \ref{df:small} there exist constants $A_i$ and $b_i$ ($i=1,2,3$)
 such that
\begin{itemize}
\item  $|e_{s_1,\beta}((t^{\frac{1}{s_1}}-\tau)^{s_1}\lambda(\zeta))|\leq A_1 e^{b_1|t^{\frac{1}{s_1}}-\tau||\zeta|^{q/s_1}}$,
\item $|E_{\Gamma_{\frac{s_2}{\kappa}}}(\zeta^{\frac{1}{\kappa}} z^{\frac{1}{\kappa}})|\leq A_2 e^{b_2|\zeta|^{1/s_2}|z|^{1/s_2}}$,
\item $|e_{\Gamma_{s_2}}(\zeta w)|\leq A_3 e^{-b_3|\zeta|^{1/s_2}|w|^{1/s_2}}$.
\end{itemize}
Hence, for fixed $w\in\CC\setminus\{0\}$
 such that $|z|$ is small relative to $|w|$ and for $|t|<a|w|^{q}$ with some fixed $a>0$ and for $\tau\in[0,t^{\frac{1}{s_1}}]$, we have
 $$
|k(t,\tau,z,w)|
 \leq
    \int_{r_0}^{\infty}\tilde{A}e^{-\tilde{b}x^{1/s_2}|w|^{1/s_2}}\,dx<
    \infty.
 $$
 It means that the integral representation of the solution $u$ is a well-defined
 holomorphic function in some complex neighbourhood of the origin. To show that this integral representation holds, observe that
 by (\ref{eq:lambda}), (\ref{eq:formal}) and (\ref{eq:integral}) we have
    \begin{eqnarray*}
     u(t,z)&=&\sum_{j=\beta-1}^{\infty}{j \choose \beta -1}
    (\partial_{\Gamma_{s_1},t}^{-1})^{j+1}\lambda^{j-\beta+1}(\partial_{\Gamma_{s_2},z})g(t,z)\\
    &=&-\sum_{j=\beta-1}^{\infty}{j \choose \beta -1}
    \int_0^{t^{\frac{1}{s_1}}}\lambda^{j-\beta+1}(\partial_{\Gamma_{s_2},z})g(\tau^{s_1},z)
    \partial_{\tau}\frac{(t^{\frac{1}{s_1}}-\tau)^{(j+1)s_1}}{\Gamma_{s_1}(j+1)}\,d\tau\\
    &=&\frac{-1}{2\kappa\pi i} \int_0^{t^{\frac{1}{s_1}}}\oint_{|w|=\varepsilon}^{\kappa} g(\tau^{s_1},w)\partial_{\tau}
     \int_{r_0e^{i\theta}}^{\infty(\theta)}\lambda^{-\beta}(\zeta)\times\\
     &&\times\sum_{j=\beta-1}^{\infty}{j \choose \beta -1}\frac{(t^{\frac{1}{s_1}}-\tau)^{s_1(j+1)}\lambda^{j+1}(\zeta)}{\Gamma_{s_1}(j+1)}
     E_{\Gamma_{\frac{s_2}{\kappa}}}(\zeta^{\frac{1}{\kappa}} z^{\frac{1}{\kappa}})
     \frac{e_{\Gamma_{s_2}}(\zeta w)}{\zeta w}\,d\zeta\,dw\,d\tau\\
     &=&\frac{-1}{2\kappa\pi i} \int_0^{t^{\frac{1}{s_1}}}\oint_{|w|=\varepsilon}^{\kappa} g(\tau^{s_1},w)\partial_{\tau}k(t,\tau,z,w)
  \,dw\,d\tau.
    \end{eqnarray*}
    The second part of the proposition is given by Theorem \ref{th:gevrey} and by the observation that
    $g,u\in\Oo_{1,1/\kappa}^{\frac{1}{s_1-qs_2}}(\CC\times D)$ if and only if $\widehat{g},\widehat{u}\in\CC[[t,z^{\frac{1}{\kappa}}]]_{qs_2-s_1,0}$.
   \end{proof}

   Using the above integral representation of solution of (\ref{eq:gamma}) we conclude that
   \begin{Th}
    \label{th:2}
    Let $m_1$, $m_2$ be moment functions of orders $s_1,s_2>0$ respectively and
    let $\lambda(\zeta)\sim\lambda_0\zeta^q$ be an analytic
    function of the variable $\xi=\zeta^{1/\kappa}$ for $|\zeta|\geq r_0$. We also assume that
    $s_1 = qs_2$, $K>0$, $d\in\RR$ and $u$ is a solution of
    \begin{equation*}
     \left\{
     \begin{array}{l}
     (\partial_{m_1,t}-\lambda(\partial_{m_2,z}))^{\beta} u=g\in\Oo_{1,1/\kappa}(D^2)\\
      \partial_{m_1,t}^j u(0,z)=0\ \textrm{for}\ j=0,\dots,\beta-1.
     \end{array}
     \right.
    \end{equation*}
    Then 
    $u(t,z)\in\Oo^{K,qK}_{1,1/\kappa}(\hat{S}_d\times\hat{S}_{(d+\arg\lambda_0+2k\pi)/q})$ for $k\in\NN$ if and only if
    $g(t,z)\in\Oo^{K,qK}_{1,1/\kappa}(\hat{S}_d\times\hat{S}_{(d+\arg\lambda_0+2k\pi)/q})$ for $k\in\NN$.
   \end{Th}
   \begin{proof}
    \par\noindent
    Let $U:=\Bo_{\overline{m}_1,t}\Bo_{\overline{m}_2,z^{1/\kappa}}u$ and
    $G:=\Bo_{\overline{m}_1,t}\Bo_{\overline{m}_2,z^{1/\kappa}}g$, where $\overline{m}_1$ and $\overline{m}_2$
    are moment functions of order $0$ defined by $\overline{m}_1(u):=\Gamma_{s_1}(u)/m_1(u)$ and
    $\overline{m}_2(u):=\Gamma_{s_2}(u)/m_2(u)$. By Propositions \ref{pr:commutation} and \ref{pr:commutation2},
    we conclude that
    $U$ is a solution of
    \begin{equation*}
     \left\{
     \begin{array}{l}
     (\partial_{\Gamma_{s_1},t}-\lambda(\partial_{\Gamma_{s_2},z}))^{\beta} U=G\in\Oo_{1,1/\kappa}(D^2)\\
      \partial_{\Gamma_{s_1},t}^j U(0,z)=0\ \textrm{for}\ j=0,\dots,\beta-1.
     \end{array}
     \right.
    \end{equation*}
    To finish the proof, by Proposition \ref{pr:sum} it is sufficient to show the equivalence
    $U(t,z)\in\Oo^{K,qK}_{1,1/\kappa}(\hat{S}_d\times\hat{S}_{(d+\arg\lambda_0+2k\pi)/q})$ for $k\in\NN$ $\Longleftrightarrow$
    $G(t,z)\in\Oo^{K,qK}_{1,1/\kappa}(\hat{S}_d\times\hat{S}_{(d+\arg\lambda_0+2k\pi)/q})$ for $k\in\NN$. 
    
    $(\Rightarrow)$ Since $G(t,z)=(\partial_{\Gamma_{s_1},t}-\lambda(\partial_{\Gamma_{s_2},z}))U(t,z)$ then $G(t,z)$ belongs to the same
    spaces $\Oo^{K,qK}_{1,1/\kappa}(\hat{S}_d\times\hat{S}_{(d+\arg\lambda_0+2k\pi)/q})$ for $k\in\NN$ as $U(t,z)$.
    
    $(\Leftarrow)$ Deforming the path of integration with respect to $w$ in the integral representation of solution given in Proposition
    \ref{pr:integral} and repeating the proof of \cite[Lemma 4]{Mic8} we conclude the assertion.
   \end{proof}

\section{Summable solution}
   Let $\widehat{u}$ be a formal solution of 
   \begin{equation}
     \label{eq:2}
     \left\{
     \begin{array}{l}
     (\partial_{m_1,t}-\lambda(\partial_{m_2,z}))^{\beta} u=\widehat{g}\in\CC[[t,z^{\frac{1}{\kappa}}]]_{\tilde{s}_1,\tilde{s_2}}\\
      \partial_{m_1,t}^j u(0,z)=0\ \textrm{for}\ j=0,\dots,\beta-1,
     \end{array}
     \right.
    \end{equation}
    where $m_1$, $m_2$ are moment functions of orders $s_1,s_2$ respectively and
    let $\lambda(\zeta)\sim\lambda_0\zeta^q$ be an analytic
    function of the variable $\xi=\zeta^{1/\kappa}$ for $|\zeta|\geq r_0$.
    
   By the Gevrey estimates $\widehat{u}\in\CC[[t,z^{\frac{1}{\kappa}}]]_{\max\{q(s_2+\tilde{s}_2)-s_1,\tilde{s}_1\},\tilde{s}_2}$.
   In the case when the Gevrey order is positive, it is natural to ask about summability of $\widehat{u}$. 
   Our aim is a characterisation of summable solutions $\widehat{u}$ in terms of the inhomogeneity $\widehat{g}$. Using Theorem \ref{th:2} we prove 
   \begin{Th}
    \label{th:simple_sum}
    Let $d\in\RR$ and $\widehat{u}$ be a formal solution of (\ref{eq:2}) with $q>0$. We consider two cases.\bigskip
    
     I. $q(s_2+\tilde{s}_2) - s_1 \geq \tilde{s}_1$, $q(s_2+\tilde{s}_2)-s_1>0$ and $s_2+\tilde{s}_2>0$. 
 In this case we assume that
    $G(t,z)=\Bo_{\Gamma_{q(s_2+\tilde{s}_2)-s_1},t}\Bo_{\Gamma_{\tilde{s}_2},z^{1/\kappa}}\widehat{g}(t,z)$ and
    $K=\frac{1}{q(s_2+\tilde{s}_2)-s_1}$. Then the following statements hold:
    \begin{itemize}
     \item If $G\in\Oo_{1,1/\kappa}^{K,qK}(\hat{S}_d\times\hat{S}_{(d+\arg\lambda_0+2k\pi)/q})$ ($k\in\NN$)
    then $\widehat{u}$ is
    $K$-summable in a direction $d$ with respect to $t$.
    \item If $G\in\Oo_{1,1/\kappa}^{K,qK}(\hat{S}_d\times\hat{S}_{(d+\arg\lambda_0+2k\pi)/q})$ ($k\in\NN$), $s_1\leq qs_2$ and $\tilde{s}_2>0$
    then $\widehat{u}$ is $(K,\frac{1}{\tilde{s}_2})$-summable in directions $(d,(d+\arg\lambda_0+2k\pi)/q)$ ($k\in\NN$).
    \item If $s_1=qs_2$ and $\tilde{s}_2>0$ then $\widehat{u}$ is $(K,\frac{1}{\tilde{s}_2})$-summable
    in directions $(d,(d+\arg\lambda_0+2k\pi)/q)$ ($k\in\NN$) if and only if $\widehat{g}$ is $(K,\frac{1}{\tilde{s}_2})$-summable
    in the same directions.
    \end{itemize}
    Moreover, if $\tilde{s}_1\leq 0$ then we may replace the condition $G\in\Oo_{1,1/\kappa}^{K,qK}(\hat{S}_d\times\hat{S}_{(d+\arg\lambda_0+2k\pi)/q})$ ($k\in\NN$)
    by $G\in\Oo_{1,1/\kappa}^{qK}(D\times\hat{S}_{(d+\arg\lambda_0+2k\pi)/q})$ ($k\in\NN$).\bigskip
    
    II. $q(s_2+\tilde{s}_2) - s_1 \leq \tilde{s}_1$, $\tilde{s}_1>0$ and $s_1+\tilde{s}_1>0$. 
 In this case we assume that
    $G(t,z)=\Bo_{\Gamma_{\tilde{s}_1,t}}\Bo_{\Gamma_{(s_1+\tilde{s}_1)/q-s_2},z^{1/\kappa}}\widehat{g}(t,z)$ and
    $K=\frac{1}{\tilde{s}_1}$. If $G\in\Oo_{1,1/\kappa}^{K,qK}(\hat{S}_d\times\hat{S}_{(d+\arg\lambda_0+2k\pi)/q})$ ($k\in\NN$) then $\widehat{u}$
    is $K$-summable in a direction $d$ with respect to $t$.
   
    Moreover, if we additionally assume that $s_1\geq q(s_2+\tilde{s}_2)$ then $\widehat{u}$ is $K$-summable in a direction $d$ if and only if
    $\widehat{g}$ has the same property.
   \end{Th}
   \begin{proof} I. We assume that $q(s_2+\tilde{s}_2) - s_1 \geq \tilde{s}_1$, $q(s_2+\tilde{s}_2)-s_1>0$ and $s_2+\tilde{s}_2>0$. 
   Applying the moment Borel transforms
   $\Bo_{\Gamma_{q(s_2+\tilde{s}_2) - s_1},t}$ and $\Bo_{\Gamma_{\tilde{s}_2},z^{1/\kappa}}$ to 
   (\ref{eq:2}) we obtain
   \begin{equation*}
     \left\{
     \begin{array}{l}
     (\partial_{m_1\Gamma_{q(s_2+\tilde{s}_2)-s_1},t}-\lambda(\partial_{m_2\Gamma_{\tilde{s}_2},z}))^{\beta}U=G\in\Oo_{1,1/\kappa}(D^2),\\
     \partial_{m_1\Gamma_{q(s_2+\tilde{s}_2)-s_1},t}^j U(0,z)=0\ \textrm{for}\ j=0,\dots,\beta-1,
     \end{array}
     \right.
   \end{equation*}
   where $U(t,z):=\Bo_{\Gamma_{q(s_2+\tilde{s}_2)-s_1},t}\Bo_{\Gamma_{\tilde{s}_2},z^{1/\kappa}}\widehat{u}(t,z)$ and
   \linebreak
   $G(t,z):=\Bo_{\Gamma_{q(s_2+\tilde{s}_2)-s_1},t}\Bo_{\Gamma_{\tilde{s}_2},z^{1/\kappa}}\widehat{g}(t,z)$.
   \par
   If $G\in\Oo_{1,1/\kappa}^{K,qK}(\hat{S}_d\times\hat{S}_{(d+\arg\lambda_0+2k\pi)/q})$ ($k\in\NN$), then by Theorem \ref{th:2}
   we conclude that also $U\in\Oo_{1,1/\kappa}^{K,qK}(\hat{S}_d\times\hat{S}_{(d+\arg\lambda_0+2k\pi)/q})$ ($k\in\NN$). Hence
   by Definition \ref{df:summable}, Remark \ref{re:summable} and Proposition \ref{pr:sum},
   $\widehat{u}\in G_{\tilde{s}_2,1/\kappa}[[t]]$ is $K$-summable in a direction $d$.
   \par
   If we additionally assume that $s_1\leq qs_2$ and $\tilde{s}_2>0$ then $qK\leq \frac{1}{\tilde{s}_2}$ and consequently
   $U\in\Oo_{1,1/\kappa}^{K,\frac{1}{\tilde{s}_2}}(\hat{S}_d\times\hat{S}_{(d+\arg\lambda_0+2k\pi)/q})$ ($k\in\NN$).
   It means by Definition \ref{df:summable2} and Proposition \ref{pr:sum} that $\widehat{u}\in\CC[[t,z^{\frac{1}{\kappa}}]]$
   is $(K,\frac{1}{\tilde{s}_2})$-summable in directions $(d,(d+\arg\lambda_0+2k\pi)/q)$ ($k\in\NN$).
   \par
   If $qK=\frac{1}{\tilde{s}_2}$ then by Theorem \ref{th:2} we conclude that
   $U(t,z)\in\Oo_{1,1/\kappa}^{K,qK}(\hat{S}_d\times\hat{S}_{(d+\arg\lambda_0+2k\pi)/q})$ ($k\in\NN$) if and only if
   $G(t,z)\in\Oo_{1,1/\kappa}^{K,qK}(\hat{S}_d\times\hat{S}_{(d+\arg\lambda_0+2k\pi)/q})$ ($k\in\NN$). By
   Definition \ref{df:summable2} and Proposition \ref{pr:sum} it means that
   $\widehat{u}\in\CC[[t,z^{\frac{1}{\kappa}}]]$ is $(K,\frac{1}{\tilde{s}_2})$-summable
    in directions $(d,(d+\arg\lambda_0+2k\pi)/q)$ ($k\in\NN$) if and only if
    $\widehat{g}\in\CC[[t,z^{\frac{1}{\kappa}}]]$ is $(K,\frac{1}{\tilde{s}_2})$-summable
    in the same directions.
    \par
    Finally, observe that $\widehat{G}\in\CC[[t,z^{\frac{1}{\kappa}}]]_{s_1+\tilde{s}_1-q(s_2+\tilde{s}_2),0}$, it means by Remark \ref{re:entire}
    that 
    $G\in\Oo_{1,1/\kappa}^{\frac{1}{q(s_2+\tilde{s}_2)-s_1-\tilde{s}_1}}(\CC\times D)$ for $q(s_2+\tilde{s}_2)>\tilde{s}_1+s_1$.
    Hence, if $\tilde{s}_1\leq 0$ then $\frac{1}{q(s_2+\tilde{s}_2)-s_1-\tilde{s}_1}\leq K$ and consequently
    $G\in\Oo_{1,1/\kappa}^{K}(\hat{S}_d\times D)$. So, in this case the conditions
    $G\in\Oo_{1,1/\kappa}^{qK}(D\times\hat{S}_{(d+\arg\lambda_0+2k\pi)/q})$ ($k\in\NN$) and
    $G\in\Oo_{1,1/\kappa}^{K,qK}(\hat{S}_d\times\hat{S}_{(d+\arg\lambda_0+2k\pi)/q})$ ($k\in\NN$) are equivalent
    \par
    II. Now we assume that $q(s_2+\tilde{s}_2) - s_1 \leq \tilde{s}_1$, $\tilde{s}_1>0$ and $s_1+\tilde{s}_1>0$.
    As in the previous case we obtain that $\widehat{u}$ is $K$-summable in a direction $d$ with respect to $t$
    under condition that $G\in\Oo_{1,1/\kappa}^{K,qK}(\hat{S}_d\times\hat{S}_{(d+\arg\lambda_0+2k\pi)/q})$ ($k\in\NN$). 
   
    Similarly to the first part of the proof, observe that in our case
    $\widehat{G}\in\CC[[t,z^{\frac{1}{\kappa}}]]_{0,s_2+\tilde{s}_2-(s_1+\tilde{s}_1)/q}$ and consequently
    $G\in\Oo_{1,1/\kappa}^{\frac{q}{s_1+\tilde{s}_1-q(s_2+\tilde{s}_2)}}(D\times \CC)$ for $q(s_2+\tilde{s}_2)<\tilde{s}_1+s_1$.
    So, if $s_1\geq q(s_2+\tilde{s}_2)$
    then also $G\in\Oo_{1,1/\kappa}^{qK}(D\times\hat{S}_{(d+\arg\lambda_0+2k\pi)/q})$ ($k\in\NN$). It means that if
    $\widehat{g}$ is $K$-summable in a direction $d$ with respect to $t$, or equivalently
    $G\in\Oo_{1,1/\kappa}^{K}(\hat{S}_d\times D)$, then $G\in\Oo_{1,1/\kappa}^{K,qK}(\hat{S}_d\times\hat{S}_{(d+\arg\lambda_0+2k\pi)/q})$ ($k\in\NN$)
    and consequently also $\widehat{u}$ is $K$-summable in a direction $d$ with respect to $t$. Conversely, if
    $\widehat{u}$ is $K$-summable in a direction $d$ with respect to $t$ then of course also $\widehat{g}=(\partial_{m_1,t}-\lambda(\partial_{m_2,z}))^{\beta}\widehat{u}$
    has the same properties.
   \end{proof}
  \begin{Ex}
   To illustrate the above theorem, let us discuss the summability of the formal solution $\widehat{u}$ of the Cauchy problem
     \begin{equation*}
     \left\{
     \begin{array}{l}
     (\partial_t-\partial^q_z)u=\widehat{g}\in\Oo[[t]]_{s},\\
     \partial_t u(0,z)=0,
     \end{array}
     \right.
   \end{equation*}
   where $q\in\NN$ and $s\in\RR$. Observe that $\widehat{u}$ satisfies
   (\ref{eq:2}) with $m_1=m_2=\Gamma_1$, $s_1=s_2=1$, $\lambda(\zeta)=\zeta^q$, $\beta=1$, $\tilde{s}_1=s$ and $\tilde{s}_2=0$.
   By Theorem \ref{th:simple_sum} we conclude that
   \begin{itemize}
    \item If $q>1$, $q-1\geq s$ and $\Bo_{\Gamma_{q-1},t}\widehat{g}(t,z)\in\Oo^{\frac{1}{q-1},\frac{q}{q-1}}(S_d\times S_{(d+2k\pi)/q})$ for
    $k=0,\dots,q-1$ then $\widehat{u}$ is $\frac{1}{q-1}$-summable in a direction $d$ with respect to $t$.
    \item If $q>1$, $s\leq 0$ and $\Bo_{\Gamma_{q-1},t}\widehat{g}(t,z)\in\Oo^{\frac{q}{q-1}}(D\times S_{(d+2k\pi)/q})$ for
    $k=0,\dots,q-1$ then $\widehat{u}$ is $\frac{1}{q-1}$-summable in a direction $d$ with respect to $t$.
    \item If $q-1\leq s$, $s>0$ and $\Bo_{\Gamma_{s},t}\widehat{g}(t,z)\in\Oo^{\frac{1}{s},\frac{q}{s}}(S_d\times S_{(d+2k\pi)/q})$ for
    $k=0,\dots,q-1$ then $\widehat{u}$ is $\frac{1}{s}$-summable in a direction $d$ with respect to $t$.
    \item If $q=1$ and $s>0$ then $\widehat{u}$ is $\frac{1}{s}$-summable in a direction $d$ with respect to $t$ if and only if $\widehat{g}$
    has the same property.
   \end{itemize}
  \end{Ex} 
 \bigskip
   
  Now, let us return to the Cauchy problem (\ref{eq:inhomo0}) with the additional condition that all $\widehat{u}_{\alpha\beta}$ constructed in
  Theorem \ref{th:gevrey} have the same Gevrey order with respect to $t$. In this case we have
 \begin{Th}
 \label{th:sum}
  Let $d\in\RR$ and $\widehat{u}$ be a formal solution of (\ref{eq:inhomo0}) determined by $\widehat{g}\in\CC[[t,z]]_{\tilde{s}_1,\tilde{s}_2}$.
  We consider two cases.\bigskip
  
  I. $s_2+\tilde{s}_2>0$ and there exist $q\in\QQ_+$  being the common pole order of $\lambda_{\alpha}(\zeta)$
  (i.e. $\lambda_{\alpha}(\zeta)\sim\lambda_{\alpha}\zeta^q$) for every
  $\alpha=1,\dots,l$ and satisfying
  $q(s_2+\tilde{s}_2)\geq s_1+\tilde{s}_1$ and $q(s_2+\tilde{s}_2)> s_1$. In this case
  we assume that 
    $G(t,z)=\Bo_{\Gamma_{q(s_2+\tilde{s}_2)-s_1},t}\Bo_{\Gamma_{\tilde{s}_2},z}\widehat{g}(t,z)$ and
    $K=\frac{1}{q(s_2+\tilde{s}_2)-s_1}$. Then the following statements hold:
    \begin{itemize}
     \item If $G\in\Oo^{K,qK}(\hat{S}_d\times\hat{S}_{(d+\arg\lambda_{\alpha}+2k\pi)/q})$ (for $\alpha=1,\dots,l$ and $k\in\NN$)
    then $\widehat{u}$ is
    $K$-summable in a direction $d$ with respect to $t$.
    \item If $G\in\Oo^{K,qK}(\hat{S}_d\times\hat{S}_{(d+\arg\lambda_{\alpha}+2k\pi)/q})$ (for $\alpha=1,\dots,l$ and $k\in\NN$),
    $s_1\leq qs_2$ and $\tilde{s}_2>0$
    then $\widehat{u}$ is $(K,\frac{1}{\tilde{s}_2})$-summable in directions $(d,(d+\arg\lambda_{\alpha}+2k\pi)/q)$ (for $\alpha=1,\dots,l$ and $k\in\NN$).
    \item If $s_1=qs_2$ and $\tilde{s}_2>0$ then $\widehat{u}$ is $(K,\frac{1}{\tilde{s}_2})$-summable
    in directions $(d,(d+\arg\lambda_{\alpha}+2k\pi)/q)$ (for $\alpha=1,\dots,l$ and $k\in\NN$) if and only if $\widehat{g}$ is
    $(K,\frac{1}{\tilde{s}_2})$-summable
    in the same directions.
    \end{itemize}
     Moreover, if $\tilde{s}_1\leq 0$ then we may replace the condition $G\in\Oo^{K,qK}(\hat{S}_d\times\hat{S}_{(d+\arg\lambda_{\alpha}+2k\pi)/q})$
      by $G\in\Oo^{qK}(D\times\hat{S}_{(d+\arg\lambda_{\alpha}+2k\pi)/q})$ (for $\alpha=1,\dots,l$ and $k\in\NN$).\bigskip
      
  II. $\tilde{s}_1>0$, $s_1+\tilde{s}_1>0$ and $q_{\alpha}(s_2+\tilde{s}_2)-s_1\leq\tilde{s}_1$ for $\alpha=1,\dots,l$, where $q_{\alpha}\in\QQ_+$ is
  a pole order of $\lambda_{\alpha}(\zeta)$ (i.e. $\lambda_{\alpha}(\zeta)\sim\lambda_{\alpha}\zeta^{q_{\alpha}}$) for every $\alpha=1,\dots,l$.
  In this case
  we assume that 
    $G_{\alpha}(t,z)=\Bo_{\Gamma_{\tilde{s}_1},t}\Bo_{\Gamma_{(s_1+\tilde{s}_1)/q_{\alpha}-s_2},z}\widehat{g}(t,z)$ and
    $K=\frac{1}{\tilde{s}_1}$. 
  If $G_{\alpha}\in\Oo^{K,q_{\alpha}K}(\hat{S}_d\times\hat{S}_{(d+\arg\lambda_{\alpha}+2k\pi)/q})$ (for $\alpha=1,\dots,l$ and $k\in\NN$)
    then $\widehat{u}$ is
    $K$-summable in a direction $d$ with respect to $t$.
   
     Moreover, if we additionally assume that $s_1\geq q_{\alpha}(s_2+\tilde{s}_2)$ (for $\alpha=1,\dots,l$)
     then $\widehat{u}$ is $K$-summable in a direction $d$ if and only if
    $\widehat{g}$ has the same property.
  \end{Th}
  \begin{proof}
  By Theorem \ref{th:gevrey} we conclude that $\widehat{u}=\sum_{\alpha=1}^l\sum_{\beta=1}^{n_{\alpha}}\widehat{u}_{\alpha\beta}$ with
   $\widehat{u}_{\alpha\beta}$ being a formal solution of simple inhomogeneous pseudodifferential equation
   \begin{equation*}
    \left\{
    \begin{array}{l}
     (\partial_{m_1,t}-\lambda_{\alpha}(\partial_{m_2,z}))^{\beta} u_{\alpha\beta}=\widehat{g}_{\alpha\beta}\\
     \partial_{m_1,t}^j u_{\alpha\beta}(0,z)=0\ \textrm{for}\ j=0,\dots,\beta-1
    \end{array}
    \right.
   \end{equation*}
   where $\widehat{g}_{\alpha\beta}(t,z):=d_{\alpha\beta}(\partial_{m_2,z})\widehat{g}(t,z)\in
   \CC[[t,z^{\frac{1}{\kappa}}]]_{\tilde{s}_1,\tilde{s}_2}$ and $d_{\alpha\beta}(\zeta)$ is a holomorphic function
   of the variable $\xi=\zeta^{\frac{1}{\kappa}}$ and of polynomial growth.
  
   In the first case, if $G_{\alpha\beta}:=\Bo_{\Gamma_{q(s_2+\tilde{s}_2)-s_1},t}\Bo_{\Gamma_{\tilde{s}_2},z}\widehat{g}_{\alpha\beta}$
   then by Proposition \ref{pr:commutation2}
   $$
   G_{\alpha\beta}=d_{\alpha\beta}(\partial_{m_2\Gamma_{\tilde{s}_2},z})\Bo_{\Gamma_{q(s_2+\tilde{s}_2)-s_1},t}\Bo_{\Gamma_{\tilde{s}_2},z}
   \widehat{g}=d_{\alpha\beta}(\partial_{m_2\Gamma_{\tilde{s}_2},z})G
   $$
   for $\alpha=1,\dots,l$ and $\beta=1,\dots,n_{\alpha}$. Hence 
   if $G\in\Oo^{K,qK}(\hat{S}_d\times\hat{S}_{(d+\arg\lambda_{\alpha}+2k\pi)/q})$ (for $\alpha=1,\dots,l$ and $k\in\NN$)
   then for every fixed $\alpha\in\{1,\dots,l\}$ and $\beta\in\{1,\dots,n_{\alpha}\}$ we conclude that
   $G_{\alpha\beta}\in\Oo_{1,1/\kappa}^{K,qK}(\hat{S}_d\times\hat{S}_{(d+\arg\lambda_{\alpha}+2k\pi)/q})$ (for $k\in\NN$).
   
   By Theorem \ref{th:simple_sum} we see that $\widehat{u}_{\alpha\beta}$ is $K$-summable in a~direction $d$ with respect to $t$ for every
   $\alpha\in\{1,\dots,l\}$ and $\beta\in\{1,\dots,n_{\alpha}\}$. Consequently also
   $\widehat{u}=\sum_{\alpha=1}^l\sum_{\beta=1}^{n_{\alpha}}\widehat{u}_{\alpha\beta}$ is $K$-summable in a~direction $d$ with respect to $t$.
   
   If additionally we assume that $s_1\leq qs_2$ and $\tilde{s}_2>0$ then by Theorem \ref{th:simple_sum} we conclude that
   $\widehat{u}_{\alpha\beta}$ is $(K,\frac{1}{\tilde{s}_2})$-summable in directions $(d,(d+\arg\lambda_{\tilde{\alpha}}+2k\pi)/q)$
   for every
   $\alpha,\tilde{\alpha}\in\{1,\dots,l\}$ and $\beta\in\{1,\dots,n_{\alpha})$, so $\widehat{u}$
   also is $(K,\frac{1}{\tilde{s}_2})$-summable in the same directions.
   
   Finally, we assume that $s_1=qs_2$ and $\tilde{s}_2>0$. If $\widehat{u}$ is $(K,\frac{1}{\tilde{s}_2})$-summable
    in directions $(d,(d+\arg\lambda_{\tilde{\alpha}}+2k\pi)/q)$ (for $\tilde{\alpha}=1,\dots,l$ and $k\in\NN$) then also
    $\widehat{g}=\tilde{P}(\partial_{m_1,t},\partial_{m_2,z})\widehat{u}$ is $(K,\frac{1}{\tilde{s}_2})$-summable
    in the same directions. For the converse, if $\widehat{g}$ is $(K,\frac{1}{\tilde{s}_2})$-summable
    in directions $(d,(d+\arg\lambda_{\tilde{\alpha}}+2k\pi)/q)$ (for $\tilde{\alpha}=1,\dots,l$ and $k\in\NN$) then also
    $\widehat{g}_{\alpha\beta}=d_{\alpha\beta}(\partial_{m_2,z})\widehat{g}(t,z)$ and, by Theorem \ref{th:simple_sum},
    $\widehat{u}_{\alpha\beta}$ are $(K,\frac{1}{\tilde{s}_2})$-summable
    in the same directions for every $\alpha=1,\dots,l$ and $\beta=1,\dots,n_{\alpha}$. It means that
    $\widehat{u}=\sum_{\alpha=1}^l\sum_{\beta=1}^{n_{\alpha}}\widehat{u}_{\alpha\beta}$ is $(K,\frac{1}{\tilde{s}_2})$-summable
    in the same directions, too.
    
    Using the same arguments as in the proof of Theorem \ref{th:simple_sum}, we finish the proof of the first case.
    In the same way we prove the second part of the theorem.
  \end{proof}
   
   \section{Multisummable solution} 
   Now we return to the general Cauchy problem (\ref{eq:inhomo0}) with $s_1,s_2,\tilde{s}_1,\tilde{s}_2\in\RR$ satisfying $s_1>0$, $s_2>0$ and
   $s_2+\tilde{s}_2>0$.
   For convenience we assume that
 \begin{equation}
   \label{eq:P_multi}
   P(\lambda,\zeta)=P_0(\zeta)\tilde{P}(\lambda,\zeta)=P_0(\zeta)\prod_{\alpha=1}^{\tilde{n}}\prod_{\beta=1}^{l_{\alpha}}
   (\lambda-\lambda_{\alpha\beta}(\zeta))^{n_{\alpha\beta}},
 \end{equation}
  where $\lambda_{\alpha\beta}(\zeta)\sim \lambda_{\alpha\beta}\zeta^{q_{\alpha}}$ is the root of the characteristic equation with
  $q_{\alpha}\in\QQ$
  and $\lambda_{\alpha\beta}\in\CC\setminus\{0\}$ for
  $\alpha=1,\dots,\tilde{n}$ and $\beta=1,\dots,l_{\alpha}$.
 Without loss of generality we may assume that there exist exactly $N$ positive pole orders of the roots of the characteristic
 equation, which are
 greater than $\frac{s_1}{s_2+\tilde{s}_2}$ and $\frac{s_1+\tilde{s}_1}{s_2+\tilde{s}_2}$,
 say $\max\{\frac{s_1}{s_2+\tilde{s}_2},\frac{s_1+\tilde{s}_1}{s_2+\tilde{s}_2}\} <q_N<\cdots<q_1<\infty$.
 Let $K_{\alpha}:=(q_{\alpha}(s_2+\tilde{s}_2)-s_1)^{-1}$ for $\alpha=1,\dots,N$.
  \par
 By Theorem \ref{th:gevrey}, the formal solution $\widehat{u}$ of (\ref{eq:inhomo0}) determined by
 $\widehat{g}\in\CC[[t,z]]_{\tilde{s}_1,\tilde{s}_2}$ is given by
  \begin{equation}
   \label{eq:decomposition}
   \widehat{u}=\sum_{\alpha=1}^{\tilde{n}}\sum_{\beta=1}^{l_{\alpha}}\sum_{\gamma=1}^{n_{\alpha\beta}}
   \widehat{u}_{\alpha\beta\gamma}
  \end{equation}
  with $\widehat{u}_{\alpha\beta\gamma}$ satisfying
  \begin{eqnarray*}
  \left\{
   \begin{array}{l}
    (\partial_{m_1,t}-\lambda_{\alpha\beta}(\partial_{m_2,z}))^{\gamma} u_{\alpha\beta\gamma}=\widehat{g}_{\alpha\beta\gamma}\\
    \partial_{m_1,t}^j u_{\alpha\beta\gamma}(0,z)=0\ \ \textrm{for}\ \ j=0,\dots,\gamma-1,
   \end{array}
  \right.
 \end{eqnarray*}
 where $\widehat{g}_{\alpha\beta\gamma}(t,z)=d_{\alpha\beta\gamma}(\partial_{m_2,z})\widehat{g}(t,z)\in
 \CC[[t,z^{\frac{1}{\kappa}}]]_{\tilde{s}_1,\tilde{s}_2}$
 and $d_{\alpha\beta\gamma}(\zeta)$ is a holomorphic function of the variable $\xi=\zeta^{\frac{1}{\kappa}}$ and of polynomial growth at infinity.
 Moreover $\widehat{u}_{\alpha\beta\gamma}\in\CC[[t,z^{\frac{1}{\kappa}}]]_{Q_{\alpha},\tilde{s}_2}$,
 where $Q_{\alpha}=\max\{q_{\alpha}^+(s_2+\tilde{s}_2)-s_1,\tilde{s}_1\}$ for $\alpha=1,\dots,\tilde{n}$. It means that
 $Q_{\alpha}=q_{\alpha}(s_2+\tilde{s}_2)-s_1$ for $\alpha=1,\dots,N$. If $N=\tilde{n}$
 then we may ask about $(K_N,\dots,K_1)$-summable solution $\widehat{u}$. Similarly, if we assume that $\tilde{s}_1\leq 0$ then
 $Q_{\alpha}\leq 0$ for $\alpha=N+1,\dots,\tilde{n}$, so also in this case we may try to find $(K_N,\dots,K_1)$-summable solution $\widehat{u}$.
 On the other hand if $\tilde{s}_1>0$ then $Q_{\alpha}=\tilde{s}_1$ for $\alpha=N+1,\dots,\tilde{n}$
 and we can ask about $(\tilde{K},K_N,\dots,K_1)$-summable solution $\widehat{u}$ with
 $\tilde{K}=\frac{1}{\tilde{s}_1}$.
 \par
 So, under the above conditions, using Theorem \ref{th:simple_sum} and similar arguments as in the proof of Theorem \ref{th:sum} we obtain
 \begin{Th}
  \label{th:multi1}
  \par
  Case I: $\tilde{s}_1\leq 0$ or $N=\tilde{n}$.
  \par
  We assume that $(d_N,...,d_1)\in\RR^N$ is an admissible multidirection
  with respect to $(K_N,\dots,K_1)$ and
  $G_{\alpha}(t,z):=\Bo_{\Gamma_{q_{\alpha}(s_2+\tilde{s}_2)-s_1},t}\Bo_{\Gamma_{\tilde{s}_2},z}\widehat{g}(t,z)$ satisfies
  \begin{equation*}
   G_{\alpha}(t,z)\in\Oo^{K_{\alpha},q_{\alpha}K_{\alpha}}
   (\hat{S}_{d_{\alpha}}\times\hat{S}_{(d_{\alpha}+\arg\lambda_{\alpha\beta}+2k\pi)/q_{\alpha}})
  \end{equation*}
  for every $\alpha=1,\dots,N$, $\beta=1,\dots,l_{\alpha}$ and $k\in\NN$. Then the formal solution $\widehat{u}$ of (\ref{eq:inhomo0}) determined
  by $\widehat{g}$ is $(K_N,\dots,K_1)$-summable in the multidirection $(d_N,\dots,d_1)$.
  \bigskip
  \par
  Case II: $\tilde{s}_1>0$ and $N<\tilde{n}$.
  \par
   We assume that $(\tilde{d},d_N,\dots,d_1)\in\RR^{N+1}$ is an admissible multidirection
   with respect to $(\tilde{K},K_N,\dots,K_1)$ and 
   $G_{\alpha}(t,z):=\Bo_{\Gamma_{q_{\alpha}(s_2+\tilde{s}_2)-s_1},t}\Bo_{\Gamma_{\tilde{s}_2},z}\widehat{g}(t,z)$ satisfies 
    \begin{equation*}
   G_{\alpha}(t,z)\in\Oo^{K_{\alpha},q_{\alpha}K_{\alpha}}
   (\hat{S}_{d_{\alpha}}\times\hat{S}_{(d_{\alpha}+\arg\lambda_{\alpha\beta}+2k\pi)/q_{\alpha}})
  \end{equation*}
  for every $\alpha=1,\dots,N$, $\beta=1,\dots,l_{\alpha}$ and $k\in\NN$.
  Additionally we assume that $q_{\alpha}>0$ and
   $G_0(t,z)=\Bo_{\Gamma_{\tilde{s}_1,t}}\Bo_{\Gamma_{\tilde{s}_2,z}}\widehat{g}(t,z)$
   satisfies
   \begin{equation*}
   G_0(t,z)\in
   \Oo^{\frac{1}{\tilde{s}_1},\frac{q_{\alpha}}{\tilde{s}_1}}(\hat{S}_{\tilde{d}}\times\hat{S}_{(\tilde{d}+\arg\lambda_{\alpha\beta}+2k\pi)/q_{\alpha}})
  \end{equation*}
   for every $\alpha=N+1\dots,\tilde{n}$, $\beta=1,\dots,l_{\alpha}$ and $k\in\NN$. Then the formal solution $\widehat{u}$ of (\ref{eq:inhomo0})
   determined by $\widehat{g}$
   is $(\tilde{K},K_N,\dots,K_1)$-summable in the multidirection $(\tilde{d},d_N,\dots,d_1)$.
 \end{Th}
 
 \section{Newton polygon}
 In this section we introduce the notion of the Newton polygon for the moment partial differential operators with constant coefficients.
 This concept is a generalisation of the Newton polygon introduced by Yonemura \cite{Y} for partial differential operators.
 The similar considerations for the partial differential operators with constant coefficients can be found in \cite[Section 3]{B5}.
 
 We will define the Newton polygon for the operator $P(\partial_{m_1,t},\partial_{m_2,z})$, where
 $$
 P(\lambda,\zeta)=\sum_{i=0}^n\sum_{j=0}^k a_{ij}\lambda^i\zeta^j.
 $$
 To this end we denote by $\Lambda$ the set of indices
 $\Lambda:=\{(i,j)\in\NN_0^2\colon a_{ij}\neq 0\}$.
 So, we may write the polynomial $P$ as
 $$
 P(\lambda,\zeta)=\sum_{(i,j)\in\Lambda} a_{ij}\lambda^i\zeta^j.
 $$
 We have
 \begin{Df}
 The \emph{Newton polygon} for the operator $P(\partial_{m_1,t},\partial_{m_2,z})$ with $s_1,s_2>0$ is defined by the convex hull of the union of sets
 $C(is_1+js_2,-i)$ ($(i,j)\in\Lambda$), that is
 \begin{equation*}
  N(P,s_1,s_2)={\rm conv\,}\{\bigcup_{(i,j)\in\Lambda} C(is_1+js_2,-i)\},
 \end{equation*}
 where $C(a,b):=\{(x,y)\in\RR^2\colon x\leq a,\ y\geq b\}$ for any $(a,b)\in\RR^2$.
  \end{Df}
  \begin{Rem}
   Observe that if $(i,j)\in\Lambda$ then the point $(is_1+js_2,-i)$ in the Newton polygon $N(P,s_1,s_2)$ is determined by the operator
   $a_{ij}\partial_{m_1,t}^i\partial_{m_2,z}^j$.
  \end{Rem}

 \begin{Rem}
 We may extend this definition to the case of pseudodifferential operators 
 \begin{gather*}
 \tilde{P}(\lambda,\zeta)=\sum_{i=0}^n a_{i}\lambda^i\tilde{P}_i(\zeta) \quad \textrm{with} \quad  \tilde{P}_i(\zeta)\sim \zeta^{q_i} \quad \textrm{and} \quad q_i\in\QQ
 \end{gather*}
 putting
 \begin{equation*}
  N(\tilde{P},s_1,s_2)={\rm conv\,}\{\bigcup_{i\in\tilde{\Lambda}} C(is_1+q_i s_2,-i)\},
 \end{equation*}
 where $\tilde{\Lambda}=\{i\in\NN_0\colon a_i\neq 0\}$. Analogously as in the previous case, if $i\in\tilde{\Lambda}$ then the point $(is_1+q_i s_2,-i)$
 is determined by the pseudodifferential operator $a_i\partial_{m_1,t}^i\tilde{P}_i(\partial_{m_2,z})$.
 \end{Rem}

 The Newton polygon allows us to describe the multisummable solutions of (\ref{eq:inhomo0}) in case of $\tilde{s}_1=0$.
 First observe that the application of the moment Borel transform $\Bo_{\tilde{m}_2,z}$, where $\tilde{m}_2$ is a moment
 function of order $\tilde{s}_2$, gives us for any operator $P$ the formula
 \begin{equation*}
 N(\Bo_{\tilde{m}_2,z}P,s_1,s_2)=N(P,s_1,s_2+\tilde{s}_2).
 \end{equation*}
 Hence, it is sufficient to consider the solutions of (\ref{eq:inhomo0}) in case of $\tilde{s}_2=0$.

 We find the shape of the Newton polygon for the operators $P(\partial_{m_1,t},\partial_{m_2,z})$
 and $\tilde{P}(\partial_{m_1,t},\partial_{m_2,z})$ given by (\ref{eq:P_multi}) with moment functions $m_1$ and $m_2$ of positive orders $s_1$ and $s_2$ respectively.
 
 So we assume that $P(\lambda,\zeta)=P_0(\zeta)\tilde{P}(\lambda,\zeta)$ and $\deg P_0(\zeta)=k\in\NN_0$. 
 By the definitions of the Newton polygons for differential and pseudodifferential operators we conclude that
 the Newton polygon $N(P,s_1,s_2)$ is created from $N(\tilde{P},s_1,s_2)$ by the translation on the vector $v=(ks_2,0)$, i.e.
 \begin{equation}
 \label{eq:newton}
 N(P,s_1,s_2)=N(\tilde{P},s_1,s_2) + (ks_2,0).
 \end{equation}
 
 Now we describe the shape of the Newton polygon for the operator
 $$
 \tilde{P}(\partial_{m_1,t},\partial_{m_2,z})=\prod_{\alpha=1}^{\tilde{n}}\prod_{\beta=1}^{l_{\alpha}}
   (\partial_{m_1,t}-\lambda_{\alpha\beta}(\partial_{m_2,z}))^{n_{\alpha\beta}},
 $$
 where $\lambda_{\alpha\beta}(\zeta)\sim \lambda_{\alpha\beta}\zeta^{q_{\alpha}}$ with 
  a pole order $q_{\alpha}\in\QQ$ and a leading term $\lambda_{\alpha\beta}\in\CC\setminus\{0\}$ for
  $\alpha=1,\dots,\tilde{n}$ and $\beta=1,\dots,l_{\alpha}$.
  As in the previous section we may assume that there exist exactly $N$ positive pole orders which are greater than $\frac{s_1}{s_2}$,
  say $\frac{s_1}{s_2}<q_N<\dots<q_1<\infty$.
  
  The set of vertices of $N(\tilde{P},s_1,s_2)$ consists of $N+1$ points $v_0,v_1,\dots,v_N$ such that
  $v_0=(ns_1,-n)$ and
  $$v_{\alpha}=\big((n-n_1-\dots-n_{\alpha})s_1+(n_1q_1+\dots+n_{\alpha}q_{\alpha})s_2, n_1+\dots+n_{\alpha}-n\big)$$
  for $\alpha=1,\dots,N$, where $n_{\alpha}=\sum_{\beta=1}^{l_{\alpha}}n_{\alpha\beta}$ is a number of formal solutions of the equation
  $\tilde{P}(\partial_{m_1,t},\partial_{m_2,z})u=0$, which are of Gevrey order $q_{\alpha}s_2-s_1$ (see Theorem \ref{th:gevrey}).
  Moreover, by (\ref{eq:newton}), $n_{\alpha}q_{\alpha}$ is an integer for $\alpha=1,\dots,N$. Observe also that
  the vertex $v_0$ is determined by $\partial_{m_1,t}^n$ and $v_{\alpha}$ ($\alpha=1,\dots,N$) by the operator
  $$
  \partial_{m_1,t}^{n-n_1-\dots-n_{\alpha}}\prod_{\gamma=1}^{\alpha}\prod_{\beta=1}^{l_{\gamma}}\lambda_{\gamma\beta}^{n_{\gamma\beta}}
  (\partial_{m_2,z}).
  $$
  
  It means that the boundary of $N(\tilde{P},s_1,s_2)$ consists of
  a horizontal half line $\Gamma_{0}$, $N$-segments $\Gamma_1,\dots,\Gamma_N$ and a vertical half line $\Gamma_{N+1}$. Let us denote by
  $K_{\alpha}$ the slope of
  $\Gamma_{\alpha}$.
  Observe that
  $K_0=0$, $K_{N+1}=\infty$ and 
  \begin{gather*}
  K_{\alpha}=\frac{n_{\alpha}}{s_2n_{\alpha}q_{\alpha}-s_1n_{\alpha}}=\frac{1}{q_{\alpha}s_2-s_1}\quad\textrm{for}\quad \alpha=1,\dots,N.
  \end{gather*}
  So $K_0=0<K_1<\dots <K_N<K_{N+1}=\infty$.
  
  Fix $\alpha\in\{1,\dots,N\}$. The set of points of segment $\Gamma_{\alpha}$ (with ends $v_{\alpha-1}$ and $v_{\alpha}$) is determined by the operator
  \begin{equation}
   \label{eq:prod}
  \Big(\partial_{m_1,t}^{n-n_1-\dots-n_{\alpha}}\prod_{\gamma=1}^{\alpha-1}\prod_{\beta=1}^{l_{\gamma}}
  \lambda_{\gamma\beta}^{n_{\gamma\beta}}(\partial_{m_2,z})\Big)
   \prod_{\beta=1}^{l_{\alpha}}
   (\partial_{m_1,t}-\lambda_{\alpha\beta}(\partial_{m_2,z}))^{n_{\alpha\beta}}.  
    \end{equation}
   Moreover the product $\partial_{m_1,t}^{n-n_1-\dots-n_{\alpha}}\prod_{\gamma=1}^{\alpha-1}\prod_{\beta=1}^{l_{\gamma}}\lambda_{\gamma\beta}^{n_{\gamma\beta}}(\partial_{m_2,z})$
   from (\ref{eq:prod})
   is a common part of all the operators determining the points of segment $\Gamma_{\alpha}$.
   So it is sufficient to consider the rest of the operator (\ref{eq:prod}), i.e. the operator
   $$
    \tilde{P}_{\Gamma_{\alpha}}(\partial_{m_1,t},\partial_{m_2,z}):= \prod_{\beta=1}^{l_{\alpha}}
   (\partial_{m_1,t}-\lambda_{\alpha\beta}(\partial_{m_2,z}))^{n_{\alpha\beta}}.
   $$  
  Let $\widehat{u}$ be a formal solution of (\ref{eq:inhomo0}) with $\tilde{s}_1=\tilde{s}_2=0$ determined by $g\in\Oo(D^2)$ and satisfying the decomposition
  (\ref{eq:decomposition}).  
  As in the previous section we conclude that $\widehat{u}_{\alpha\beta\gamma}$ is a formal series of Gevrey order $1/K_{\alpha}$ with respect to $t$, which satisfies
  $\tilde{P}_{\Gamma_{\alpha}}(\partial_{m_1,t},\partial_{m_2,z})u_{\alpha\beta\gamma}= g_{\alpha\beta\gamma}\in\Oo_{1,1/\kappa}(D^2)$ (for every
  $\beta\in\{1,\dots,l_{\alpha}\}$ and $\gamma\in\{1,\dots,n_{\alpha\beta}\}$). 
  
  So, the Gevrey order $1/K_{\alpha}$ of $\widehat{u}_{\alpha\beta\gamma}$ is determined by the slope of $\Gamma_{\alpha}$ and $K_{\alpha}$-summability 
  is determined by the operator $\tilde{P}_{\Gamma_{\alpha}}(\partial_{m_1,t},\partial_{m_2,z})$.
  
\bibliographystyle{siam}
\bibliography{summa}
\end{document}